\documentclass[11pt]{amsart}

\usepackage{amsmath,amsthm,amssymb}
\usepackage{graphicx}
\usepackage[all,cmtip]{xy}
\usepackage{url}
\usepackage{fixltx2e}
\usepackage{mathrsfs} %(Allows fancy scripts; use \mathscr)
\usepackage[left=1.2in,right=1.2in,top=1.1in,bottom=1.1in]{geometry} %(change margin sizes)

\newtheorem{thm}{Theorem}
\newtheorem{cor}[thm]{Corollary}
\newtheorem{lem}[thm]{Lemma}
\newtheorem{prop}[thm]{Proposition}
\newtheorem*{defn}{Definition}

\theoremstyle{remark}
\newtheorem*{remark}{Remark}

\numberwithin{equation}{section}
\numberwithin{thm}{section}

\newenvironment{packed_item}{
\begin{itemize}
  \setlength{\itemsep}{5pt}
  \setlength{\parskip}{0pt}
  \setlength{\parsep}{0pt}
}{\end{itemize}}

\DeclareMathOperator{\spc}{ }

\DeclareMathOperator{\Gal}{Gal}
\DeclareMathOperator{\Hom}{Hom}

\DeclareMathOperator{\End}{End}

\DeclareMathOperator{\Aut}{Aut}

\newcommand{\To}{\longrightarrow}
\newcommand{\tab}{\spc\spc\spc\spc\spc\spc\spc\spc}

\newcommand{\C}{\mathbb{C}}
\newcommand{\Q}{\mathbb{Q}}
\newcommand{\Z}{\mathbb{Z}}

\newcommand{\set}[1]{\left\{#1\right\}}
\newcommand{\setst}[2]{\left\{\begin{array}{l|r} #1 & #2 \end{array}\right\}}

\newcommand{\inn}[1]{\left\langle#1\right\rangle}

\newcommand{\symE}{{E^{(3)}}}
\newcommand{\symEone}{{E_1^{(3)}}}
\newcommand{\symbigE}{{\mathcal E^{(3)}}}

\newcommand{\ps}{\mathbb{P}}
\newcommand{\Dsys}{|\mathfrak D|}
\newcommand{\lss}[2]{{}^{#1}\textrm{$#2$}}

\DeclareMathOperator{\bigO}{O}

\DeclareMathOperator{\Bl}{Bl}
\DeclareMathOperator{\rank}{rank}
\DeclareMathOperator{\AJ}{AJ}
\DeclareMathOperator{\Alb}{Alb}

\DeclareMathOperator{\SO}{SO}

\DeclareMathOperator{\Sing}{Sing}
\DeclareMathOperator{\alg}{alg}

\DeclareMathOperator{\Spin}{Spin}
\DeclareMathOperator{\GSpin}{GSpin}
\DeclareMathOperator{\CH}{CH}

\DeclareMathOperator{\an}{an}
\DeclareMathOperator{\KS}{KS}
\DeclareMathOperator{\pr}{pr}

\newcommand{\enmot}{\mathscr{E}nd}
\newcommand{\cc}{CC\textsubscript{3} }

\title[The Tate Conjecture for a family of surfaces of general type]{The Tate Conjecture for a family of surfaces of general type with $p_g=q=1$ and $K^2=3$}
\author{Christopher Lyons}
\address{California State University, Fullerton \\
Department of Mathematics \\
800 N. State College Blvd \\
Fullerton, CA 92834 }
\email{clyons@fullerton.edu}

\begin{document}

\begin{abstract}
We prove a big monodromy result for a smooth family of complex algebraic surfaces of general type, with invariants $p_g=q=1$ and $K^2=3$, that has been introduced by Catanese and Ciliberto.  This is accomplished via a careful study of degenerations.  As corollaries, when a surface in this family is defined over a finitely generated extension of $\Q$, we verify the semisimplicity and Tate conjectures for the Galois representation on the middle $\ell$-adic cohomology of the surface.
\end{abstract}

\maketitle

\setcounter{tocdepth}{1}
\tableofcontents

\medskip

% 1111111111111111111111111

\section{Introduction}

\noindent In this paper, all fields under consideration will be subfields of $\C$.

In the Enriques classification of algebraic surfaces, those of general type are far less understood than their counterparts of nonmaximal Kodaira dimension.  This is not only true geometrically, but arithmetically as well.  In particular, given that surfaces of general type are, in some sense, the most common among all surfaces, they offer an important testing ground for a number of well-known, wide open conjectures in arithmetic geometry.

A second class of arithmetically interesting surfaces are those with geometric genus one.  Via the Hodge structure on their middle singular cohomology groups, these surfaces are related to objects that have traditionally received more attention in arithmetic geometry, namely abelian varieties, K3 surfaces, and Shimura varieties.  In particular, the relation with abelian varieties, first discovered by Kuga and Satake \cite{KS}, is \emph{a priori} only of a transcendental nature, but one expects it to be algebraic in light of the Hodge Conjecture.  This expectation suggests the possibility of transferring known arithmetic results for abelian varieties to surfaces of geometric genus one, an idea first explored by Deligne \cite{Del-K3}.

With this in mind, we focus on a class of surfaces of general type that also have geometric genus one.  More specifically, they are minimal algebraic surfaces with geometric genus $p_g=1$, irregularity $q=1$, self-intersection number $K^2=3$ of the canonical divisor, and Albanese fiber genus $g=3$.   These surfaces have been introduced and classified by Catanese and Ciliberto \cite{CC}.  For this reason, they have been called \emph{Catanese-Ciliberto surfaces of fiber genus three} by Ishida \cite{Ish}, but for brevity they will be referred to in this paper simply as \emph{\cc surfaces}.
  
When its canonical divisor $K$ is ample, we will call a \cc surface \emph{admissible}.  Catanese and Ciliberto construct a smooth projective family $\pi:\mathcal X\to\mathcal S$ over $\Q$ containing all admissible \cc surfaces, where $\mathcal S$ is a smooth irreducible variety of dimension $5$.  Our first theorem concerns the monodromy representation of the topological fundamental group $\pi_1(\mathcal S(\C),\sigma)$ on the second singular cohomology of the fiber $\mathcal X_\sigma$; to state it, we need some notation.

Every \cc surface has two numerically independent curves, one being the canonical divisor $K$ and the other a smooth Albanese fiber $F$.  In the family $\pi_\C:\mathcal X_\C\to\mathcal S_\C$, the cycle classes of $K$ and $F$ in $H^2(\mathcal X_\sigma,\Q)(1)$ both come from global sections of the local system $\mathbb H:= R^2 (\pi_\C^{{\an}})_\ast \Q(1)$.  Moreover, the first of these two global sections provides a polarization for the family.  Denote by $\mathbb V$ the polarized variation of rational Hodge structure over $\mathcal S_\C$ that one obtains by taking the orthogonal complement in $\mathbb H$ of these two global sections with respect to cup product.  Let $\phi_\sigma$ denote the polarization on $\mathbb V_\sigma$.  Then the Hodge structure $\mathbb V_\sigma$ is of dimension 9 and type $\set{(-1,1),(0,0),(1,-1)}$, with a polarization of signature $(2,7)$.  

Our first theorem is a big monodromy result for the family $\pi_\C:\mathcal X_\C\to\mathcal S_\C$:

\medskip

\noindent\textbf{Theorem A.} \emph{
For $\sigma\in\mathcal S(\C)$, the image of the monodromy representation
\[
\Lambda: \pi_1(\mathcal S(\C),\sigma) \to \bigO(\mathbb V_\sigma,\phi_\sigma)
\]
is Zariski-dense.
}
\medskip

Now suppose one has an admissible \cc surface $X$ defined over a finitely generated subfield $k_0$ of $\C$.  The aforementioned work of Kuga and Satake gives a Hodge correspondence between $X_\C$ and a certain complex abelian variety, called its Kuga-Satake variety.  Using Theorem A, one can show that this abelian variety has a model $A$ over a finite extension $k_0'$ of $k_0$.  Moreover, one can show that the Hodge correspondence between $X_\C$ and its Kuga-Satake variety arises from a \emph{motivated} correspondence in the sense of Andr\'e \cite{Andre-Mot} (and hence an \emph{absolute Hodge} correspondence in the sense of Deligne \cite{Del-AH}) between $X_{k_0'}$ and $A$.  Using this and work of Faltings \cite{Fal, Fal-Wus}, one obtains the following:

\medskip

\noindent\textbf{Theorem B.} \emph{Let $k_0$ be a finitely generated subfield of $\C$, let $k$ be its algebraic closure, and let $X$ be an admissible \cc surface defined over $k_0$.  For a prime number $\ell$, consider the $\ell$-adic representation
\[
r_\ell:\Gal(k/k_0)\to \Aut\bigl( H^2(X_{k},\Q_\ell)(1) \bigr).
\]
Then the following hold:
\begin{itemize}
\item[(i)] The representation $r_\ell$ is semisimple.
\item[(ii)]  (Tate Conjecture) Let $V_{\alg}$ be the $\Q_\ell$-subspace generated by the image of the cycle class map
\[
c_\ell:\CH^1(X_{k})\to H^2(X_{k},\Q_\ell)(1).
\]
Then $V_{\alg}$ is exactly the subspace of elements in $H^2(X_{k},\Q_\ell)(1)$ that are fixed by an open subgroup of $\Gal(k/k_0)$.
\end{itemize}
}

\medskip

We remark that the method described above to obtain Theorem B from Theorem A has been known to experts for some time, see \cite[p.80]{Tate-Motives}.  The prototype is the case of K3 surfaces, which follows from work of Deligne \cite{Del-K3}.  Andr\'e \cite{Andre-KS} axiomatizes this strategy and applies it not only to K3 surfaces, but also to abelian surfaces, a class of surfaces of general type appearing in \cite{Cat, Tod}, and cubic fourfolds (where one is concerned with the cohomology group $H^4$ rather than $H^2$).  In deducing Theorem B from Theorem A, we follow the proof laid out in \cite{Andre-KS}, but at certain steps we must account for one key difference.  In each of the aforementioned cases, the proof of the big monodromy theorem is obtained as a corollary of the following fact: the image of the period map of the family in question contains a Euclidean open ball in the period domain.  Such a proof is unavailable in the case of \cc surfaces, as the dimension of moduli is $5$ and the dimension of the period domain is $7$.

Instead, we proceed by an analysis of the degenerations of admissible \cc surfaces.  In their classification of \cc surfaces, Catanese and Ciliberto show that if one fixes an elliptic curve $E$, the admissible \cc surfaces $X$ with $\Alb(X)=E$ are exactly the smooth divisors in a certain complete linear system $|\mathfrak D|$ on the symmetric cube $\symE$ of $E$.  From the construction of the family $\pi:\mathcal X\to \mathcal S$, the proof of Theorem A will follow if there exists one elliptic curve $E$ for which one can prove that the monodromy of the family of all smooth divisors in $|\mathfrak D|$ has dense image.

Since $\mathfrak D$ is not very ample on $\symE$, classical Lefschetz theory does not apply directly, but a mild generalization results in a number of hypotheses that, if satisfied, will give the proof.  The most difficult hypotheses to verify concern the structure of the collection of singular elements in $|\mathfrak D|$: they assume that this collection has exactly one irreducible component of codimension one in $|\mathfrak D|$ and that the general singular element has a singular locus of just one ordinary double point.    We use equations for \'etale covers of elements of $|\mathfrak D|$ given by Ishida \cite{Ish} to show that this holds whenever a suitably nice pencil exists in $\Dsys$.  Using the computer program \textsc{Singular}, we then verify the existence of such a pencil when $E$ is the elliptic curve is
\[
y^2=x^3+x^2-59x-783/4.
\]

\medskip

As a final remark, let us mention one simple consequence of Theorem A.  The Picard number of any \cc surface lies between 2 and 9.  Examples of Xiao \cite{Xiao} realize the upper bound.  Since the family $\pi:\mathcal X\to\mathcal S$ is defined over $\Q$, one may use Theorem A and a general result of Andr\'e \cite[Thm 5.2(3)]{Andre-Mot} to conclude:  

\medskip

\noindent\textbf{Corollary.} \emph{There exist \cc surfaces over $\bar\Q$ with minimal Picard number $2$.}

\medskip

Here is a brief outline of the paper. In \S\ref{prelim_sec} we give background on \cc surfaces.  We carry out our analysis of the singular elements of $\Dsys$ on $\symEone$ in \S\ref{degen_sec} and use it in \S\ref{mon_sec} to prove Theorem A.  Finally, Theorem B is proved within a broader axiomatic framework in \S\ref{Tate_sec}.

We end by noting that the source code and data files for all computer calculations used in this paper (specifically, in Proposition \ref{computations}) are available online at \cite{LyData}.

\medskip

\noindent\textbf{Acknowledgments.} This paper is based upon my Ph.D.\ thesis.  I thank my advisor, Dinakar Ramakrishnan, for suggesting this topic and for his guidance throughout the project.  Thanks are also due to Tom Graber, Kapil Paranjape, and Chad Schoen for helpful suggestions and  observations, to Don Blasius, Jordan Ellenberg and Kartik Prasanna for comments on earlier versions of this paper, and to Bhargav Bhatt, Amin Gholampour, and Baptiste Morin for answering questions.  Finally, I thank the referee for the corrections and improvements they have suggested.

This work has been partially supported by NSF RTG  grant DMS 0943832.

\medskip

\noindent\textbf{Conventions and notations:}

All fields considered will be subfields of $\C$.  We will use $k$ to denote an algebraically closed field.  Other conventions about fields are found at the beginning of \S3 and \S5.

Let $V$ be a vector space over $k$.  If $v\in V\setminus\set{0}$, we let $\bar v = k^\times \cdot v\in \mathbb P(V)$ denote its equivalence class in the projectivization of $V$.

For a variety $X$ with invertible sheaf $\mathcal L$ and global section $\sigma\in H^0(X,\mathcal L)$, let $Z(\sigma)$ denote the divisor of zeros of $\sigma$.  Since $Z(\sigma)$ depends only upon the class $\bar\sigma\in \mathbb P H^0(X,\mathcal L)$, we write $Z(\bar \sigma)=Z(\sigma)$.

In reference to a collection of objects parametrized by an algebraic variety, the adjective \emph{general} is used to refer to any member of the collection that lies within a sufficiently small Zariski-dense open subset of the parameter space (e.g., a general fiber of a fibration or a general pencil in a linear system).

% 2222222222222222222222222

\section{Preliminaries on \cc surfaces}\label{prelim_sec}

Recall that $k$ denotes an algebraically closed subfield of $\C$.

\begin{defn}
A minimal smooth projective connected algebraic surface $X$ over $k$ will be called a \emph{Catanese--Ciliberto surface of fiber genus three}, or more succinctly a \emph{\cc surface}, if all of the following hold:
\begin{packed_item}
\item $X$ has geometric genus $p_g=h^0(X,\Omega_X^2)=1$,
\item $X$ has irregularity $q=h^1(X,\mathcal O_X)=1$,
\item the canonical divisor $K$ of $X$ has self-intersection number $K^2=3$, and
\item the generic fiber of an Albanese fibration $\Alb:X\to\Alb(X)$ has genus $g=3$.
\end{packed_item}
We call $X$ \emph{admissible} if, additionally, $K$ is ample (or equivalently, if its canonical model is smooth).

\end{defn}

These invariants imply that a \cc surface is of general type (e.g., see Table 10 in \cite{BPV}).

If $E$ is an elliptic curve over $k$, let $\bigoplus$ denote the addition law, and let $0\in E(k)$ denote the identity.  Let $\symE$ denote the symmetric cube of $E$, which is the quotient of $E^3$ by the $S_3$-action permuting the factors; let $q:E^3\to \symE$ denote this quotient.  The summation map $\bigoplus: E^3\to E$ is $S_3$-invariant, and hence induces the \emph{Abel-Jacobi map} $\AJ: \symE\to E$.  One may define two divisors on $\symE$.  The first is denoted by $D_0$ and defined as the scheme-theoretic image $q(\set{0}\times E\times E)$ in $\symE$, and the second is $G_0 := \AJ^{-1}(0)$.  More intuitively, the points of $\symE$ correspond to effective divisors $\sum n_P P$ of degree 3 on $E$ and the Abel-Jacobi map
\begin{eqnarray*}
\AJ &:& \symE \To E \\
&& \sum n_P P\mapsto \bigoplus n_P P,
\end{eqnarray*}
sends an effective divisor of degree 3 on $E$ to the sum of its components.  The divisor $D_0$ represents those effective divisors of degree 3 whose support contains the point $0$.

The Abel-Jacobi map $\AJ: \symE\to E$ gives $\symE$ the structure of a $\mathbb P^2$-bundle over $E$.  The tautological invertible sheaf of this $\mathbb P^2$-bundle is $\mathcal O_{\symE}(D_0)$, and $G_0$ is a fiber.  We set 
\begin{equation}\label{D_defn}
\mathfrak D:=4D_0-G_0
\end{equation}
and
\begin{equation}\label{L_defn}
L:=\mathcal O_{\symE}(\mathfrak D).
\end{equation}

\begin{thm}[Catanese--Ciliberto \cite{CC}]\label{classification}
Let $X$ be a \cc surface over $k$ and let $E=\Alb(X)$.  Then the canonical model of $X$ is isomorphic to a divisor in the linear system $|\mathfrak D|$ on $\symE$ with at most rational double points as singularities.

Conversely, if $E$ is any elliptic curve over $k$, then any element of the linear system $|\mathfrak D|$ on $\symE$ with at most rational double points is the canonical model of a \cc surface with Albanese variety $E$.  Moreover, a general element $X\in |\mathfrak D|$ is smooth.  For such $X$, the restrictions from $\symE$ to $X$ of the divisors $D_0$ and $G_0$ are, respectively, numerically equivalent to the canonical divisor $K$ and an Albanese fiber $F$.
\end{thm}

Catanese and Ciliberto show that
\begin{equation}\label{dim_D_sys}
h^0(\symE,L)=5.
\end{equation}
Thus the canonical models of all \cc surfaces $X$ over $k$ with $\Alb(X)=E$ belong to a family over an open subset of $\ps H^0(\symE,L) \simeq \ps^4$.  The total space of this family is an open subset of the incidence correspondence
\[
\setst{ (q,\sigma) \in \symE\times \ps H^0(\symE,L) }{ q\in Z(\sigma) }.
\]

By relativizing this picture, Catanese and Ciliberto construct a single connected family containing the canonical models of all \cc surfaces, and hence a subfamily containing all admissible \cc surfaces, as follows.  For a fixed $N>3$, let $Y:=Y_1(N)$ be the (open) modular curve classifying elliptic curves with $\Gamma_1(N)$-level structure.  Then $Y$ is a connected fine moduli space with universal family $\mathcal E\to Y$, which is defined over $\Q$ using Shimura's canonical model \cite{Shim}.  Let $\mathcal Y\stackrel{O}{\to} \mathcal E$ denote the identity section.

Then $S_3$ acts upon $\mathcal E\times_Y \mathcal E\times_Y \mathcal E$ as a $Y$-scheme by permuting the factors, and the resulting quotient is the relative symmetric cube $p:\mathcal E^{(3)}\to Y$ of $\mathcal E\to Y$; from the $S_3$-invariance of the threefold addition map $\bigoplus: \mathcal E\times_Y \mathcal E\times_Y \mathcal E \to \mathcal E$ one obtains the relative Abel-Jacobi map $\mathcal A\mathcal J: \mathcal E^{(3)}\to \mathcal E$.  One may then define on $\mathcal E^{(3)}$ the divisor $\mathcal D_0$ as the scheme-theoretic image of the composition
\[
Y\times_Y \mathcal E\times_Y\mathcal E \To \mathcal E\times_Y \mathcal E\times_Y \mathcal E \To \mathcal E^{(3)},
\]
where the first map arises from base changing the identity section $Y\stackrel{O}{\to} \mathcal E$ via the projection $\pr_1:  \mathcal E\times_Y \mathcal E\times_Y \mathcal E \to \mathcal E$, and the second map is the quotient by $S_3$.  One may also define on $\mathcal E^{(3)}$ the divisor $\mathcal G_0$ as the fiber product
\[
\xymatrix{
\mathcal G_0 \ar[r] \ar[d] & \mathcal E^{(3)} \ar[d]^{\mathcal A\mathcal J} \\
Y \ar[r]^{O} & \mathcal E.
}
\]
Note that if $y\in Y(k)$ is such that $\mathcal E_y\simeq E$, then $\mathcal D_0$ and $\mathcal G_0$ restrict to $D_0$ and $G_0$ on on the fiber $\mathcal E_y^{(3)}\simeq \symE$. One may form the invertible sheaf
\[
\mathcal L:=\mathcal O_{\mathcal E^{(3)}}(4\mathcal D_0-\mathcal G_0)
\]
on $\mathcal E^{(3)}$ and the projective bundle $\mathcal S_0:=\ps(p_\ast\mathcal L)$ over $Y$.  Then $\mathcal S_0$ is a $\ps^4$-bundle over $Y$, and one may identify its fiber over the point $y$ as
\[
(\mathcal S_0)_y \simeq \ps H^0(\symE,L).
\]
By forming the incidence correspondence
\[
\mathcal X_0:=\set{(q,s) \in \mathcal E^{(3)}\times_Y \mathcal S_0 \ | \ q\in Z(s)}
\]
and projecting onto the second factor, one obtains a flat projective family $\mathcal X_0\to \mathcal S_0$ whose fibers consist exactly of the divisors in $|\mathfrak D|$ on $\symE$, for all possible choices of $E$.  By restricting to the locus of smooth fibers, one obtains in this way a smooth geometrically connected variety $\mathcal S$ of dimension $5$ and a smooth projective family $\pi:\mathcal X\to\mathcal S$ over $\Q$.

By Theorem \ref{classification}, one has:

\begin{cor}\label{CCfamily}
Over any point of $\mathcal S(k)$, the fiber of the smooth projective family $\pi:\mathcal X\to\mathcal S$ is an admissible \cc surface over $k$.  Conversely, every admissible \cc surface over $k$ is isomorphic to the fiber over a point of $\mathcal S(k)$.
\end{cor}

We finish this section by recording the following useful result.

\begin{thm}[Xiao, Polizzi]\label{CM_exists}
There exist admissible complex \cc surfaces $X$ having maximal Picard number $9$.  
\end{thm}

\begin{proof}
For any $E$, Xiao \cite[p.51, Cor.\ 4]{Xiao} constructs a \cc surface $X$ with a genus two pencil and $\Alb(X)=E$, and Polizzi \cite[Props.\ 6.3, 6.18]{Pol} shows it is admissible and has maximal Picard number.
\end{proof}

% 3333333333333333333333333

\section{Singular elements in $\Dsys$}\label{degen_sec}

Throughout \S\ref{degen_sec}, the base field will always be $\C$.

\subsection{}\label{Ishida_subsec}

Let $E$ be a complex elliptic curve and let $\Dsys$ be the complete linear system on $\symE$ defined in (\ref{D_defn}).  We describe here certain \'etale threefold covers of the elements of $\Dsys$ that are  amenable to computation, which in turn allow one to deduce certain facts about $\Dsys$.  These techniques are used by Ishida in \cite{Ish} to study Albanese fibrations of \cc surfaces, and draw from more general situations considered in \cite{Tak}.  We refer to either of these sources for further details and references.

The two key motivating facts are:
\begin{enumerate}

\item The Abel-Jacobi map $\AJ:\symE\to E$ makes $\symE$ into a $\ps^2$-bundle $\mathbb P(B)\to E$, where $B$ is an indecomposable locally free sheaf of rank 3 and degree 1.
\item Let $\tilde E$ be an elliptic curve with identity $\tilde 0\in\tilde E$ such that there exists an isogeny $\varphi:\tilde E\to E$ of degree 3.  Then $B':=\varphi_\ast \mathcal O_{\tilde E}(\tilde 0)$ is an indecomposable locally free sheaf of rank 3 and degree 1 with the property that $\varphi^\ast B'$ is a direct sum of three invertible sheaves on $\tilde E$.
\end{enumerate}
As $\symE=\mathbb P(B)$ and $\mathbb P(B')$ are isomorphic \cite{Ati}, we may fix an identification between the two.  Defining the $\ps^2$-bundle
\begin{equation}\label{tildeP_defn}
\tilde P:=\mathbb P(\varphi^\ast B'),
\end{equation}
with projection $\tilde \AJ:\tilde P\to \tilde E$, one obtains a commutative diagram
\[
\xymatrix{
\tilde P \ar[r]^{\Phi} \ar[d]^{\tilde \AJ} & \symE \ar[d]^{\AJ} \\
\tilde E \ar[r]^{\varphi} & E
}
\]
in which $\tilde P$ is the fiber product of $\tilde E$ and $\symE$ over $E$.  Thus, if we let $G=\ker\varphi=\set{\tilde 0,C_1,C_2}$, then both $\varphi$ and $\Phi$ are Galois coverings with group $G$.  More specifically, if $Q\in\tilde E$ and $\tau_Q\in\Aut(\tilde E)$ denotes translation by $Q$, then $\gamma\in G$ acts on $\tilde E$ by $\tau_\gamma$ and on $\tilde P$ by the base change $\tilde\tau_\gamma$ of $\tau_\gamma$.

The locally free sheaf $\varphi^\ast B'$ splits as
\begin{equation}\label{splitting}
\varphi^\ast B' \simeq \mathcal O_{\tilde E}(\tilde 0)\oplus \mathcal O_{\tilde E}(C_1)\oplus \mathcal O_{\tilde E}(C_2),
\end{equation}
from which one obtains
\begin{eqnarray*}
H^0(\tilde P,\mathcal O_{\tilde P}(1)) &\simeq& H^0(\tilde E,\mathcal O_{\tilde E}(\tilde 0))\oplus H^0(\tilde E,\mathcal O_{\tilde E}(C_1))\oplus H^0(\tilde E,\mathcal O_{\tilde E}(C_2)).
\end{eqnarray*}
Let $Z_0$ (resp., $Z_1,Z_2$) denote the section in $H^0(\tilde P,\mathcal O_{\tilde P}(1))$ that corresponds to the constant function $1$ in $H^0(\tilde E,\mathcal O_{\tilde E}(\tilde 0))$ (resp., $H^0(\tilde E,\mathcal O_{\tilde E}(C_1))$, $H^0(\tilde E,\mathcal O_{\tilde E}(C_2))$).  The splitting  (\ref{splitting}) allows one to easily obtain local trivializations of $\tilde P$ over $\tilde E$, two of which we make explict:
\begin{itemize}
\item Consider the open subset $U = \tilde E\setminus G = \tilde E\setminus\set{\tilde 0,C_1,C_2}$.  Each of the three invertible sheaves $\mathcal O_{\tilde E}(\tilde 0)$, $\mathcal O_{\tilde E}(C_1)$, $\mathcal O_{\tilde E}(C_2)$ are isomorphic over $U$ to the constant sheaf $\mathcal O_{\tilde E}$ via the identity map.  It follows from (\ref{tildeP_defn}) and (\ref{splitting}) that $(Z_0\colon Z_1\colon Z_2)$ serve as relative homogenous coordinates for the restriction $\tilde P|_{U}$ of the $\mathbb P^2$-bundle $\tilde P$ to $U$.
\item The rational function $t=x/y$ on $\tilde E$ is a local parameter at $\tilde 0$.  Let $U'\subseteq \tilde E$ denote the complement of $C_1, C_2$, and the three nontrivial 2-torsion points of $\tilde E$; that is, $U'$ is the open subset of $\tilde E$ obtained by removing all points in the support of $\text{div}(t)$ except for $\tilde 0$.  Over $U'$, the sheaves $\mathcal O_{\tilde E}(C_1)$ and $\mathcal O_{\tilde E}(C_2)$ are still isomorphic to $\mathcal O_{\tilde E}$ via the identity map.  For $\mathcal O_{\tilde E}(\tilde 0)$ one may use the isomorphism $\mathcal O_{\tilde E}(\tilde 0)|_{U'}\tilde\To \mathcal O_{\tilde E}|_{U'}$ given by multiplication by $t$.  Setting $Z_0':=t^{-1}Z_0$, it follows that $(Z_0'\colon Z_1 \colon Z_2)$ serve as relative homogenous coordinates for $\tilde P|_{U'}$ over $U'$
\end{itemize}

Note that the action of $G$ on $Z_0,Z_1,Z_2$ is described by
\begin{equation}\label{Gaction1}
\tilde \tau_{C_1}^*Z_0=Z_2, \tab \tilde \tau_{C_2}^*Z_0=Z_1.
\end{equation}
Together with the open set $\tilde P\vert_{U}$ (which is stabilized by $G$), the three $G$-translates of $\tilde P\vert_{U'}$ form a trivializing open cover of $\tilde P$.  As it happens, we will only concern ourselves with local properties of $G$-stable divisors on $\tilde P$, and so will only need to utilize the two coordinate charts described above.

Choose an affine equation $y^2=w(x)$ for $\tilde E$, where $w$ is a monic cubic polynomial with nonzero discriminant.  If $C_1=(\alpha,\beta)$ (and thus $C_2=(\alpha,-\beta)$), define three rational functions on $\tilde E$ by
\begin{eqnarray*}
f &:=& x-\alpha, \\
g &:=& x\circ \tau_{C_2}-\alpha, \\
h &:=& x\circ \tau_{C_1}-\alpha.
\end{eqnarray*}
From their definitions, $G$ permutes these three functions:
\begin{equation}\label{Gaction2}
g=\tau_{C_2}^*f, \tab h=\tau_{C_1}^*f.
\end{equation}

With the invertible sheaf $L$ as in (\ref{L_defn}), one has isomorphisms
\[
\Phi^\ast L \simeq \mathcal O_{\tilde P}(4)\otimes \mathcal O_{\tilde P}(-\tilde P_{\tilde 0}-\tilde P_{C_1}-\tilde P_{C_2})
\]
and
\begin{equation}\label{Phi_isom}
\Phi^*:H^0(\symE,L)\tilde \To H^0(\tilde P,\Phi^*L)^G.
\end{equation}

\begin{prop}[\cite{Ish,Tak}] \label{IshidaEqns}
A basis for the space $H^0(\tilde P,\Phi^*L)^G$ is given by the following sections:
\begin{eqnarray*}
\Psi_1 &:=& fZ_0^4+gZ_1^4+hZ_2^4, \\
\Psi_2 &:=& Z_0Z_1Z_2(Z_0+Z_1+Z_2), \\
\Psi_3 &:=& fZ_0^3Z_2+gZ_1^3Z_0+hZ_2^3Z_1, \\
\Psi_4 &:=& fZ_0^3Z_1+gZ_1^3Z_2+hZ_2^3Z_0, \\
\Psi_5 &:=& ghZ_1^2Z_2^2+fhZ_0^2Z_2^2+fgZ_0^2Z_1^2.
\end{eqnarray*}
\end{prop}

Consider a global section $s\in H^0(\tilde P,\Phi^*L)^G$ and the $G$-stable divisor $Z(s)$ that it determines on $\tilde P$.  As remarked above, in order to study local properties of $Z(s)$, it suffices to study it only on the two open sets $\tilde P|_U$ and $\tilde P|_{U'}$.  Given that $(Z_0\colon Z_1\colon Z_2)$ form relative homogenous coordinates on the open set $\tilde P|_{U}$, the equations given in Proposition \ref{IshidaEqns} are well-suited for analyzing the divisor $Z(s)$ on $\tilde P|_{U}$.  On the other hand, making the substitution $Z_0=t Z_0'$ and dividing by $t$ allows one to analyze  $Z(s)$ on the open set $\tilde P|_{U'}$, where relative homogenous coordinates are $(Z_0':Z_1:Z_2)$; define
\[
\chi_i\bigl(t,Z_0',Z_1,Z_2\bigr):=t^{-1}\Psi_i(tZ_0',Z_1,Z_2)
\]
for $1\leq i\leq 5$.  If one expands the rational functions $f,g,h$ in terms of the local parameter $t$ (and defines $\mu:=w'(\alpha)$), one obtains the following expressions for $\chi_i$:
\begin{eqnarray*}
\chi_1 &=& 2\beta(Z_1^4-Z_2^4)+t(Z_0'^4+\mu Z_1^4+\mu Z_2^4)+O(t^2), \\
\chi_2 &=& Z_0'Z_1Z_2(Z_1+Z_2)+tZ_0^2Z_1Z_2, \\
\chi_3 &=& Z_0'^3Z_2-2\beta Z_1Z_2^3+t(\mu Z_1Z_2^3+2\beta Z_0'Z_1^3)+O(t^2),\\
\chi_4 &=& Z_0'^3Z_1+2\beta Z_1^3Z_2+t(\mu Z_1^3Z_2-2\beta Z_0'Z_2^3)+O(t^2),\\
\chi_5 &=& 2\beta Z_0'^2(Z_1^2-Z_2^2)+t(\mu Z_0'^2Z_2^2+\mu Z_0'^2Z_1^2-4\beta^2Z_1^2Z_2^2)+O(t^2).
\end{eqnarray*}
These expressions reveal the behavior of $Z(s)$ in a neighborhood of the fiber $\tilde P_{\tilde 0}$, which is the key information about $Z(s)$ on the open set $\tilde P|_{U'}$ that cannot be obtained by looking on $\tilde P|_{U}$.

Recall that our primary focus is the divisors in the linear system $\Dsys$ on $\symE$.  By (\ref{Phi_isom}), the pullback of any divisor in $\Dsys$ to $\tilde P$ is an \'etale threefold cover of the form $Z(s)$, for some $s\in H^0(\tilde P,\Phi^*L)^G$; thus local properties of elements of $\Dsys$ may be determined by instead looking at such $Z(s)$ on $\tilde P$.  This method is adopted to carry out several computations below.  In order to label sections of $H^0(\symE,L)$ in what follows, we use the isomorphism (\ref{Phi_isom}) to define
\begin{equation}\label{psi_defn}
\psi_i:=(\Phi^\ast)^{-1}(\Psi_i)\in H^0(\symE,L).
\end{equation}

Before proceeding, we demonstrate the utility of the equations of Ishida given in Proposition \ref{IshidaEqns} by resolving a question about the set of base points of $\Dsys$, which is conjectured in \cite{CC} to be empty.

\begin{prop}\label{bp}
There are exactly four base points of $|\mathfrak D|$, each of which is simple and belongs to the fiber $G_0\subseteq\symE$.
\end{prop}

\begin{proof}
It is shown in \cite[Lemma 3.3]{CC} that any possible base points of $|\mathfrak D|$ must be simple, while in \cite[Theorem 3.8]{Pol} it is shown that there are at most four base points, each of which lies in the fiber $G_0$.  Hence it suffices to show that the divisors $Z(s)$, for $s\in H^0(\tilde P,\Phi^\ast L)^G$, all pass through four common points in the fiber $\tilde P_{\tilde 0}$.   If ${r_1,r_2,r_3}$ are the roots of $X^3-2\beta$ (note that $\beta\neq 0$ since $C_1$ is 3-torsion), one finds that the four values
\[
(1:0:0), \ (r_1:1:-1), \ (r_2:1:-1), \ (r_3:1:-1)
\]
for $(Z_0':Z_1:Z_2)$ are zeros of each $\chi_i$ when $t=0$.
\end{proof}

%%%%%%%%%%%%%%

\subsection{}\label{pencil_subsec}
We now analyze the collection of singular divisors in $\Dsys$.  To start, define
\[
S:=\ps H^0(\symE,L)\simeq\ps^4
\]
and let
\begin{equation}\label{R_defn}
R:=\set{s\in S \ | \ Z(s)\text{ is singular}} \subseteq S,
\end{equation}
which parametrizes the singular divisors in $\Dsys$.  As the general element of $\Dsys$ is smooth by Theorem \ref{classification}, $R$ is a proper Zariski-closed subset of $S$.  Endow $R$ with its unique structure as a reduced subscheme of $S$.

\begin{lem}\label{smooth_axis_lemma}
Choose any $s_0\in S$ such that $Z(s_0)$ is smooth, and let $J\subseteq S$ be a general pencil (i.e., a general one-dimensional projective subspace of $S$) that contains $s_0$.  Then the base locus of $J$ is smooth.
\end{lem}

\begin{proof}
The base locus of such a pencil may be regarded as an element of the trace $\mathfrak t$ on $Z(s_0)$ of the linear system $\Dsys$.  Since $Z(s_0)\in\Dsys$, the linear system $\mathfrak t$ on $Z(s_0)$ has exactly the same four base points as does $\Dsys$ on $\symE$.  By Bertini's theorem applied to $\mathfrak t$, we conclude that the base locus of a general pencil through $s_0$ is smooth away from the four base points in Proposition \ref{bp}.

Now consider one of these base points $q\in \symE$.  By the aforementioned proposition, $q$ is a simple base point of $\Dsys$, which is equivalent to saying that two general elements of $\Dsys$ are smooth at $q$ and intersect transversally there.  It follows that, for general $s_1\in S\setminus\set{s_0}$, $Z(s_1)$ is smooth at $q$ and intersects $Z(s_0)$ transversally, which implies that the base locus of the pencil generated by $s_0$ and $s_1$ is smooth at $q$.

Thus the base locus of a general pencil $J$ through $s_0$ is smooth both at and away from the four base points.
\end{proof}

\begin{lem}\label{euler_comp_lem}
Let $J\subseteq S$ be a pencil with smooth base locus $\Gamma$, and consider the blow-up
\[
\mathcal Y:=\Bl_{\Gamma}(\symE)=\set{(q,s)\in\symE\times J \ | \ q\in Z(s)}.
\]
of $\symE$ along $\Gamma$.  If $e$ denotes the topological Euler characteristic, then we have
\[
e(J)e(Z(s))-e(\mathcal Y)=42
\]
for general $s\in J$.
\end{lem}

\begin{proof}
Since $\mathcal Y=\Bl_{\Gamma}(\symE)$, it follows from \cite[pp.605--606]{GH} that $e(\mathcal Y)=e(\symE)+e(\Gamma)$.  We have $e(\symE)=0$ \cite{Mac} and, regarding $\Gamma$ as a smooth curve on $Z(s)\simeq \mathcal Y_s$ for a general $s\in J$, the adjunction formula gives
\[
-e(\Gamma)=2g(\Gamma)-2=\Gamma.(\Gamma+K),
\]
where $K$ is the canonical divisor on $Z(s)$.  If $\iota:Z(s)\hookrightarrow\symE$ denotes the embedding then, up to numerical equivalence, we have $\Gamma=\iota^\ast \mathfrak D=\iota^*(4D_0-G_0)$ and $K=\iota^*(D_0)$ \cite[p.395]{CC}.  Since $Z(s)$ is numerically equivalent to $4D_0-G_0$ in $\symE$, we calculate
\begin{eqnarray*}
\Gamma.(\Gamma+K) &=& \iota^*\bigl((4D_0-G_0).(5D_0-G_0)\bigr) \\
&=& \iota^*(20D_0^2-9D_0.G_0) \\
&=& (4D_0-G_0).(20D_0^2-9D_0.G_0) \\
&=& 24.
\end{eqnarray*}
Thus $e(\mathcal Y)=-24$.

Finally, we have $e(J) = e(\ps^1)=2$ and, as $Z(s)$ is an admissible \cc surface by Theorem \ref{classification}, Noether's formula gives $e(Z(s))=9$.  Therefore
\[
e(J)e(Z(s))-e(\mathcal Y)=2\cdot 9-(-24)=42.
\]
\end{proof}

\begin{cor}\label{dimR_cor}
We have $\dim R=3$.
\end{cor}

\begin{proof}
Let $J\subseteq S$ be a general pencil with base locus $\Gamma\subseteq \symE$.  Then the total space of $J$ is the blow-up $\mathcal Y$ defined in Lemma \ref{euler_comp_lem}, whose second projection we denote by $\rho: \mathcal Y\to J$.

Clearly $\dim R\leq 3$, so suppose for contradiction that $\dim R\leq 2$.  Then for a general pencil $J$, the total space $\mathcal Y$ is smooth by Lemma \ref{smooth_axis_lemma} and $\rho:\mathcal Y\to J$ is a smooth fibration.  This implies
\[
e(\mathcal Y) = e(J)e(Z(s))
\]
for any $s\in J$, contradicting Lemma \ref{euler_comp_lem}.
\end{proof}

Given a point $q\in\symE$, consider the collection of $s\in R$ such that $q\in\Sing Z(s)$.  This collection is a projective subspace of $S$, which is seen as follows.  Let $(x_1,x_2,x_3)$ be local parameters of $\symE$ at $q$ and let $f_i$ be the local equation for $\psi_i$ in a neighborhood of $q$.  Then the divisor $Z(a_1\psi_1+\ldots+a_5\psi_5)$ has a singular point at $q$ if and only if the column vector $(a_1,\ldots,a_5)^t$ belongs to the kernel of the matrix
\begin{equation}\label{sing_matrix}
\left[
\begin{array}{ccc}
f_1(q) & \ldots & f_5(q) \\
(\partial f_1/\partial x_1)(q) & \ldots & (\partial f_5 /\partial x_1)(q) \\
(\partial f_1/\partial x_2)(q) & \ldots & (\partial f_5 /\partial x_2)(q) \\
(\partial f_1/\partial x_3)(q) & \ldots & (\partial f_5 /\partial x_3)(q) \\
\end{array}
\right].
\end{equation}
Note that the kernel of this matrix is independent of the choice of local parameters $x_i$.  The projectivization of this kernel is then the projective subspace of $S$ parametrizing elements with a singularity at $q$.

\begin{lem}
There exists a point $q_0\in \symE$ for which there is exactly one $s\in R$ satisfying $q_0\in\Sing Z(s)$.
\end{lem}

\begin{proof}
Lifting the problem to the covering space $\tilde P$, it suffices to prove that there is a point $\tilde q_0\in \tilde P$ for which there is a unique $\tilde s\in \ps (H^0(\tilde P,\Phi^\ast L)^G)$ such that $\tilde q_0\in \Sing Z(\tilde s)$.

Consider the point $\tilde q_0 = (t,(Z_0':Z_1:Z_2))=(0,(0:1:1))\in\tilde P_{\tilde 0}$.  Letting $u=Z_0'/Z_2$, $v=Z_1/Z_2-1$, the three functions $(t,u,v)$ form local parameters of $\tilde P$ at $\tilde q_0$.  In these parameters, the local equation of $\Psi_i$ is
\[
g_i = \chi_i(t,u,v+1,1).
\]
In this case the analogue of the matrix (\ref{sing_matrix}) is
\[
\left[
\begin{array}{ccc}
g_1(\tilde q_0) & \ldots & g_5(\tilde q_0) \\
(\partial g_1/\partial t)(\tilde q_0) & \ldots & (\partial g_5 /\partial t)(\tilde q_0) \\
(\partial g_1/\partial u)(\tilde q_0) & \ldots & (\partial g_5 /\partial u)(\tilde q_0) \\
(\partial g_1/\partial v)(\tilde q_0) & \ldots & (\partial g_5 /\partial v)(\tilde q_0) \\
\end{array}
\right] 
 = 
 \left[
\begin{array}{ccccc}
0 & 0 & -2\beta & 2\beta & 0 \\
2\mu & 0 & \mu & \mu & -4\beta^2 \\
0 & 2 & 0 & 0 & 0\\
8\beta & 0 & -2\beta & 6\beta & 0 \\
\end{array}
\right],
\]
and $Z(a_1\Psi_1+\ldots a_5\Psi_5)$ is singular at $\tilde q_0$ if and only if $(a_1,\ldots,a_5)^t$ belongs to the kernel.  One calculates the kernel to be the span of $(-2\beta^2, 0, 4\beta^2,4\beta^2, \mu)^t$.
\end{proof}

In light of the preceding lemma, there exists a rational map
\[
\begin{array}{ccccl}
\eta &:& \symE &\dashrightarrow& R \\
&& q&\mapsto& (\text{the unique $s$ such that $Z(s)$ is singular at $q$}).
\end{array}
\]
More algebraically, this map is defined using the matrix (\ref{sing_matrix}).  If $m_i$ denotes the 4-by-4 minor of this 4-by-5 matrix obtained by omitting the $i$th column, and one or more $m_i$ is nonzero, then we define
\begin{equation}\label{lambda_def}
\eta(q)=\overline{\sum_{i=1}^5 (-1)^i m_i \psi_i} \in R,
\end{equation}
which is equal to the projectivization of the one-dimensional kernel.  The rational map $\eta$ is defined on an open subset of $\symE$, and we let $\hat R$ denote the Zariski-closure of its image in $R$ under $\eta$.  Since $\symE$ is irreducible, so is $\hat R$.

\subsection{} We now distinguish certain pencils in $S$.  As the definition and terminology will indicate, the pencils in question have several properties in common with Lefschetz pencils of hyperplane sections.

\begin{defn}
We will say that a pencil $J\subseteq S$ is an \emph{L-pencil} if it satisfies each of the following: 
\begin{enumerate}
\item[(L1)] The base locus in $\symE$ of $J$ is smooth.
\item[(L2)] For all $s\in J$, the divisor $Z(s)$ contains at most isolated singularities, and $Z(s)$ is smooth for general $s\in J$.
\item[(L3)] There are at least $42$ values of $s\in J$ such that $Z(s)$ is singular.
\end{enumerate}
\end{defn}

\begin{prop}\label{Lprop}
Let $J\subseteq S$ be an L-pencil.  Then there are in fact exactly $42$ values of $s\in J$ such that $Z(s)$ is singular and, for each such $s$, the singular locus of $Z(s)$ consists of one ordinary double point.
\end{prop}

\begin{proof}
Consider the total space $\mathcal Y:=\Bl_{\Gamma}(\symE)$ with its projection $\rho:\mathcal Y\to J$.
By (L1), $\mathcal Y$ is smooth.  Those points $(q,s)\in\mathcal Y$ such that $q\in\Sing Z(s)$ form the collection of critical points of $\rho$; thus $\rho$ has finitely many critical points $\set{y_i}$ by (L2).  If $y_i=(q_i,s_i)$ is such a critical point, let $\mu_i$ denote the Milnor number of the isolated singularity that $Z(s_i)$ possesses at the point $q_i$.  Recall that $\mu_i$ is a positive integer with the property that $\mu_i=1$ if and only if the singularity is an ordinary double point.  We have $\sum_i \mu_i \geq 42$ by (L3).  To prove the proposition, it will suffice therefore to show that $\sum_i \mu_i = 42$.

According to \cite[14.1.5(d)]{Ful} there is a certain zero-cycle $\gamma$ on $\mathcal Y$ satisfying each of the following:
\begin{enumerate}
\item $\gamma$ is supported on the set of critical points of $\rho:\mathcal Y\to J$.
\item Let $y_i\in\mathcal Y$ be a critical point of $\rho$.  Then the restriction of $\gamma$ to $\set{y_i}$ is the zero-cycle $\mu_i y_i$.
\item One has
\[
\deg(\gamma)=e(J)e(Z(s))-e(\mathcal Y),
\]
where $s\in J$ is a general.
\end{enumerate}
Thus we obtain
\[
e(J)e(Z(s))-e(\mathcal Y)=\sum_i \mu_i.
\]
Applying Lemma \ref{euler_comp_lem} gives $\sum_i \mu_i=42$.
\end{proof}

\begin{prop}\label{more_Lprop}
Suppose there exists an L-pencil $J\subseteq S$ with the following property: Whenever $q\in \Sing Z(s)$ for some $s\in J$, the rational map $\eta:\symE\dashrightarrow R$ is defined at $q$.  Then the following hold:
\begin{enumerate}
\item[(a)] The irreducible component $\hat R$ is a hypersurface of degree $42$ in $S$.
\item[(b)] The only $3$-dimensional irreducible component of $R$ is $\hat R$.
\item[(c)] For all $s$ in a Zariski-dense open subset of $\hat R$, $\Sing Z(s)$ consists of one ordinary double point.
\end{enumerate}
\end{prop}

\begin{proof}
If for some $s\in J$ the divisor $Z(s)$ is singular at $q\in\symE$, then by assumption $s$ belongs to the image of the rational map $\eta$.  Hence one may consider the fiber $\eta^{-1}(s)$, which is necessarily a subset of $\Sing Z(s)$.  Since $J$ is an L-pencil, $\Sing Z(s) = \set{q}$ by Proposition \ref{Lprop}, and we conclude that there exist 0-dimensional fibers of $\eta$.  This implies, by semi-continuity of the fiber dimension, that $\dim \hat R=\dim \symE = 3$.  Thus $\hat R$ is a hypersurface in $S$.

Next consider the correspondence
\[
T = \set{ (q,s)\in \symE\times R \ | \ q\in\Sing Z(s) }
\]
with second projection $\pr_2:T\to R$, which is necessarily surjective.  We endow $T$ with its reduced subscheme structure in $\symE\times R$.   Let $R'$ be any  3-dimensional irreducible component of $R$.  Since $\pr_2$ is surjective, the component $R'$ must be dominated by one or more irreducible components of $T$.  If $T'$ is such a component, then $T'$ is closed in $T$, so by properness of $\pr_2$ its image is also closed; thus $\pr_2|_{T'}: T'\to R'$ is surjective.  As $R'$ is a hypersurface in $S$, the given L-pencil $J$ intersects it; at such an intersection point $\sigma$, the fiber $\pr_2^{-1}(\sigma) = \Sing Z(\sigma)\times\set{\sigma}$ is 0-dimensional by Proposition \ref{Lprop}.  It follows by semi-continuity of fiber dimension that, for general $s\in R'$, $\Sing Z(s)\simeq\Sing Z(s)\times\set{s}  = \pr_2^{-1}(s)$ is 0-dimensional.

Thus we have established that the divisors belonging to a general pencil $J'\subseteq S$ have at most isolated singularities, since $J'$ is in general position with respect to $R$.  Moreover, $J'$ must have at least 42 values of $s$ for which $Z(s)$ is singular; indeed, by the given hypothesis on the given L-pencil $J$, we may conclude
\[
\deg R \geq \deg \hat R \geq 42.
\]
(Here $\deg R$ denotes the sum of the degrees of all $3$-dimensional irreducible components of $R$ as hypersurfaces in $S$.)  Finally, the base locus of $J'$ is smooth by Lemma \ref{smooth_axis_lemma}.  In summary, this means that the general pencil $J'$ will be an L-pencil, and therefore by Proposition \ref{Lprop} we may conclude that $\deg R=\deg \hat R=42$.  Thus $\hat R$ is the only 3-dimensional component of $R$.  This proves (a) and (b).

To finish, choose a point $s_0\in S\setminus R$, so that $Z(s_0)$ is smooth.  After choosing an identification between $\mathbb P^3$ and the set of pencils in $S$ passing through $s_0$, we obtain a finite map $\ell: \hat R\to \mathbb P^3$ sending a point in $\hat R$ to the pencil containing itself and $s_0$.  Then, by Lemma \ref{smooth_axis_lemma} and the preceding paragraphs, there is a dense open subset of $\mathbb P^3$ consisting of pencils that have smooth base locus and are in general position with respect to $\hat R$, i.e., a dense open subset of L-pencils.  By Proposition \ref{Lprop}, $\Sing Z(s)$ consists of one ordinary double point whenever $s$ belongs to the inverse image under $\ell$ of this open subset.  This proves (c).
\end{proof}

\subsection{}

We now exhibit a specific elliptic curve $E_1$ for which one may verify the hypotheses of Proposition \ref{more_Lprop}.  Let $E_1$ denote the curve given by the Weierstrass equation
\begin{eqnarray*}
E_1 &:& y^2=x^3+x^2-59x-783/4
\end{eqnarray*}
and let $\tilde E_1$ denote
\begin{equation*}
\tilde E_1\ : \ y^2=w(x):=x^3+x^2+x-3/4.
\end{equation*}
Then there is an isogeny $\varphi:\tilde E_1\to E_1$ such that $\ker\varphi=\set{\tilde 0,C_1,C_2}$, where $C_1=(\alpha,\beta):=(1,3/2)$.
Consider the pencil $J_1\subseteq S$ defined parametrically by
\begin{equation}\label{J1_defn}
J_1:=\set{ a\psi_1+b(\psi_3-\psi_4)\ | \ (a:b)\in\ps^1}.
\end{equation}

\begin{prop}\label{computations}
The pencil $J_1$ of divisors on $\symEone$ is an L-pencil.  Moreover, if $q\in \Sing Z(s)$ for $s\in J_1$, then the rational map $\eta: \symEone\dashrightarrow R$ is defined at $q$.

\end{prop}

\begin{proof}
To verify each of the properties (L1) -- (L3) for $J_1$, it suffices to instead work on the \'etale cover $\Phi:\tilde P\to\symEone$.  Specifically, it suffices to show that the corresponding pencil of pullbacks
\[
\tilde J_1 = \set{ a\Psi_1+b(\Psi_3-\Psi_4)\ | \ (a:b)\in\ps^1},
\]
on $\tilde P$ has smooth base locus, that it has at least 42 singular elements (but only finitely many), and that each of these singular elements has only isolated singularities.

One may verify this with the aid of the computer program \textsc{Singular} by working, separately for both open sets $\tilde P|_U$ and $\tilde P_{U'}$, with explicit rational polynomial rings and ideals coming from the equations in \S\ref{Ishida_subsec}.  First one checks that the ideal corresponding to the singularities of the base locus is empty.  Next one examines the subset
\begin{equation}\label{126cps}
\set{(\tilde q,\tilde s)\in \tilde P \times \tilde J_1 \ | \ \tilde q\in\Sing Z(\tilde s)},
\end{equation}
by showing that the associated quotient ring over open subsets have Krull dimension 0 and then by computing their dimensions as vector spaces over $\Q$.  One finds that the set (\ref{126cps}) consists of 126 points, and that its projection to $\tilde J_1\simeq \ps^1$ is parametrized by an irreducible (hence separable) polynomial of degree 42.  This allows one to conclude that there are exactly 42 singular elements in $\tilde J_1$, and (by $G$-symmetry) that the singular locus of each of these consists of three points.

From \textsc{Singular} one also obtains high-precision approximations to the complex coordinates of each point in the (\ref{126cps}).  Inputting these numerical coordinates into a program such as \textsc{Mathematica} allows one verify that $\eta$ is defined at each of the 126 points of $(\ref{126cps})$ (which of course amounts to checking that a matrix analogous to (\ref{sing_matrix}) has at least one nonzero minor).

Details of these computations may be found online \cite{LyData}.
\end{proof}

By combining Propositions \ref{more_Lprop} and \ref{computations}, we summarize the main outcome of \S\ref{degen_sec}:

\begin{thm}\label{sec3_thm}
Let $E_1$ be the elliptic curve
\[
E_1: \ y^2=x^3+x^2-59x-783/4,
\]
let $S=\ps H^0(\symEone,L)\simeq\ps^4$ parametrize the linear system $\Dsys$ on $\symEone$, and let $R\subseteq S$ be the reduced subscheme parametrizing singular elements.  Then
\begin{enumerate}
\item[(a)] $\dim R=3$.
\item[(b)] $R$ has exactly one $3$-dimensional irreducible component $\hat R$, which is a hypersurface of degree $42$ in $S$.
\item[(c)] For all $s$ in a Zariski-dense open subset of $\hat R$, $\Sing Z(s)$ consists of one ordinary double point.
\end{enumerate}
\end{thm}

\begin{remark}
The choice of the particular elliptic curve $E_1$ is computationally motivated.  Namely, the cover $\tilde E_1$ and its 3-torsion point $C_1$ are selected first for the simplicity of their coefficients, so that certain Gr\"obner basis computations are more efficient, and $E_1$ is simply the resulting quotient $\tilde E_1/\inn{C_1}$.
\end{remark}

% 4444444444444444444444444

\section{Big monodromy}\label{mon_sec}

\subsection{}\label{VHS_subsec}

Recall the smooth projective family $\pi:\mathcal X\to \mathcal S$ over $\Q$ constructed in \S\ref{prelim_sec} that contains all admissible \cc surfaces.  Via the embedding $\mathcal X\hookrightarrow \symbigE\times_Y \mathcal S$, the divisors $\mathcal D_0\times_Y\mathcal S$ and $\mathcal G_0\times_Y\mathcal S$ cut out divisors on $\mathcal X$ that we denote by $\lss{D}{\mathcal X}$ and $\lss{G}{\mathcal X}$.  For each $s\in\mathcal S(k)$, the divisors $\lss{D}{\mathcal X}$ and $\lss{G}{\mathcal X}$ in turn cut out divisors on the admissible \cc surface $\mathcal X_s$ that are numerically equivalent to the canonical divisor $K$ and an Albanese fiber $F$.  Define the local system $$\mathbb H_\Z:=R^2 (\pi_\C^{\an})_\ast \Z(1)/(\text{tors.})$$ on $\mathcal S_\C$, letting $\mathbb H:=\mathbb H_\Z\otimes\Q$.  The divisors $\lss{D}{\mathcal X}$, $\lss{G}{\mathcal X}$ yield respective classes $\xi_1$, $\xi_2$ in $H^0(\mathcal S_\C,\mathbb H_{\Z})$.

\begin{lem}\label{global_ind}
The global sections $\xi_1,\xi_2\in H^0(\mathcal S_\C,\mathbb H_{\Z})$ are linearly independent.
\end{lem}

\begin{proof}
It suffices to show the restrictions of these sections to the fiber $H^2(\mathcal X_\sigma,\Z)(1)/(\text{tors.})$ of $\mathbb H_{\Z}$ at a point $\sigma\in\mathcal S(\C)$ are independent.  But these restrictions are just the cycle classes of an Albanese fiber and the canonical divisor on $\mathcal X_s$, which are independent by the adjunction formula.
\end{proof}

Let
\[
\phi: \mathbb H_\Z\otimes\mathbb H_\Z\to R^4(\pi_\C^{\an})_\ast \Z(2)\simeq \Z
\]
be the cup product form.  Let $\mathbb V_\Z$ be the orthogonal complement under $\phi$ of the rank 2 local subsystem of $\mathbb H_\Z$ generated by $\set{\xi_1,\xi_2}$ and let $\mathbb V:=\mathbb V_\Z\otimes\Q$.  Then $\mathbb V_\Z$ naturally gives rise to an integral variation of Hodge structure, for which we use the same notation, of weight zero and rank 9 with Hodge numbers $h^{1,-1}=h^{-1,1}=1, h^{0,0}=7$.  Moreover, the global section $\xi_1$ of $\mathbb H_\Z$ restricts to the class of the canonical divisor in each fiber $\mathbb H_{\Z,\sigma}=H^2(\mathcal X_\sigma,\Z)(1)/(\text{tors.})$, which is ample since each $\mathcal X_s$ is an admissible \cc surface.  Thus the cup product form $\phi:\mathbb V_\Z\otimes\mathbb V_\Z\to \Z$ makes $\mathbb V_\Z$ into a polarized integral variation of Hodge structure on $\mathcal S_\C$.

Recalling the construction of $\pi:\mathcal X\to \mathcal S$, suppose the point $y\in Y(\C)$ on the modular curve $Y$ corresponds to the elliptic curve $E_1$ from Theorem \ref{sec3_thm}.  Then $\mathcal S_{0,y}\simeq S=\ps H^0(\symEone,L)$.  Under such an identification, let $J\subseteq \mathcal S_{0,y}$ be a general line representing a pencil in $|\mathfrak D|$, and let $J^\ast\subseteq J$ represent the locus of smooth divisors.  Hence the pullback of $\mathcal X\to\mathcal S$ to $J^\ast$ is the restriction of the total space $\mathcal Y\to J$ of the pencil to the smooth fibers.  In order prove Theorem A regarding the monodromy of $\mathbb V$ over $\mathcal S_C$, it will suffice to prove the following:

\begin{thm}\label{supp_thm}For $\sigma\in J^\ast$, the Zariski-closure of the image of the monodromy representation
\begin{equation}\label{rhoL}
\lambda: \pi_1(J^\ast(\C),\sigma)\to \bigO(\mathbb V_{\C,\sigma},\phi_\sigma)
\end{equation}
is the complex algebraic group $\bigO(\mathbb V_{\C,\sigma},\phi_\sigma)$.
\end{thm}

In summary, we are reduced to investigating the monodromy of a general pencil of divisors in the complete linear system $\Dsys$ on $\symEone$.

%%%%%%%%%%%%%%%%%%%%%%%%%%%

\subsection{}\label{Thm_A_subsec}

The proof of Theorem \ref{supp_thm} will follow along the lines of the classical theory of Lefschetz, which concerns the setting of Lefschetz pencils of very ample divisors on a smooth projective variety.  Our primary reference for Lefschetz theory will be the exposition \cite{Lamot}.  One may also consult \cite[Expos\'es XVII, XVIII]{SGA7}.

The divisor $\mathfrak D$ on $\symEone$ is ample (by \cite[Prop.\ 1.14]{CC}) but not very ample (e.g., by Proposition \ref{bp}), so Lefschetz theory does not apply directly to our case.  However, with Theorem \ref{sec3_thm} at our disposal, the techniques in \cite{Lamot} can be extended to the present case.  Since Lefschetz theory is well-known and  the extension is straightforward, we only highlight the most salient points in the argument.

As mentioned, $J^\ast$ is the smooth locus of a general pencil $J$ in the linear system $\Dsys$ on $\symEone$.  By Lemma \ref{smooth_axis_lemma} and Theorem \ref{sec3_thm}, $J$ has the following properties:
\begin{enumerate}
\item The base locus of $J$ is smooth.
\item If $Z(s)$ is singular for some $s\in J$, then $\Sing Z(s)$ consists of one ordinary double point.
\item $J$ intersects the singular locus $R\subseteq S$ only at smooth points of $\hat R$, and does so transversally.
\end{enumerate}
Fix a base point $\sigma\in J^\ast(\C)$, so that $Z(\sigma)$ is smooth.  Then (1) and (2) allow the application of Picard-Lefschetz theory to the fibration $\mathcal Y\to J$.  In particular, this theory introduces a collection of vanishing cycles in $H_2(\mathcal Y_\sigma,\Q)(-1)$ (of which there are 42 by Theorem \ref{sec3_thm}).  By abuse of terminology, we will also use the term \emph{vanishing cycles} to refer to the elements in $H^2(\mathcal Y_\sigma,\Q)(1)$ that are the Poincar\'e duals of the (homological) vanishing cycles; since our primary viewpoint will be cohomological, this should not result in confusion.

\begin{lem}
Under the isomorphism $\mathbb H_\sigma \simeq H^2(\mathcal Y_\sigma,\Q)(1)$, the subspace $\mathbb V_\sigma$ of $\mathbb{H}_\sigma$ is identified as a $\pi_1(J^\ast(\C),\sigma)$-module with the subspace of $H^2(\mathcal Y_\sigma,\Q)(1)$ generated by the vanishing cycles of the fibration $\mathcal Y\to J$.  In particular, $\pi_1(J^\ast(\C),\sigma)$ does not fix any nonzero vector in $\mathbb V_\sigma$.
\end{lem}

\begin{proof}
The isomorphism of $\pi_1(J^\ast(\C),\sigma)$-modules
\[
\mathbb H_\sigma \simeq H^2(\mathcal Y_\sigma,\Q)(1)
\]
identifies the spans of the canonical and Albanese classes on each side, and thus identifies their orthogonal complements.  In $\mathbb H_\sigma$ this complement is $\mathbb V_\sigma$.  Thus it will suffice to show that the complement in $H^2(\mathcal Y_\sigma,\Q)(1)$ is the span of the vanishing cycles.

Let $I\subseteq H^2(\mathcal Y_\sigma,\Q)(1)$ denote the space of $\pi_1(J^\ast(\C),\sigma)$-invariants.  Then the subspace spanned by the canonical and Albanese classes in $H^2(\mathcal Y_\sigma,\Q)(1)$ is contained in $I$, and we claim this inclusion is an equality.  Using Deligne's theorem of the fixed part, $I$ is exactly the image of $H^2(\mathcal Y,\Q)(1)$ induced by $\mathcal Y_\sigma\hookrightarrow \mathcal Y$.  Since $\mathcal Y$ is the blow-up of $\symEone$ along the smooth base locus, one has
\[
H^2(\mathcal Y,\Q)(1)\simeq H^2(\symEone,\Q)(1)\oplus \Q(1),
\]
where the summand $\Q(1)$ is generated by the class of the exceptional divisor.  We know that the summand $H^2(\symEone,\Q)(1)$ contributes the canonical and Albanese classes to $H^2(\mathcal Y_\sigma,\Q)(1)$.  Furthermore, the class of the exceptional divisor restricts on $\mathcal Y_\sigma$ to that of the base locus curve, which is generated by the canonical and Albanese classes.  This shows the claim.

Next, using the fact that $\mathfrak D$ is ample (and therefore an analogue of Hard Lefschetz holds for $\symEone$ using the class of $\mathfrak D$), it follows that
\begin{equation}\label{PL_direct_sum}
H^2(\mathcal Y_\sigma,\Q)(1) = I\oplus V,
\end{equation}
where $V$ is the kernel of the Gysin map $H^2(\mathcal Y_\sigma,\Q)(1) \to H^2(\symEone,\Q)(1)$ (see \cite[p.412]{PS}).  But by Picard-Lefschetz theory, $V$ is the subspace generated by the vanishing cycles of the fibration $\mathcal Y\to J$.  Since $V\subseteq I^{\perp}$ by the Picard-Lefschetz formula, (\ref{PL_direct_sum}) gives $V=I^\perp$.  In particular, no nonzero vector in $V$ is fixed since $I\cap V=0$.
\end{proof}

\begin{prop}\label{single_orbit}
Concerning the action of $\pi_1(J^\ast(\C),\sigma)$ on $\mathbb V_{\sigma}$, we have:
\begin{enumerate}
\item[(a)] Let $\delta\in \mathbb V_\sigma$ be a vanishing cycle; then for any other vanishing cycle $\delta' \in \mathbb V_\sigma$, either $\delta'$ or $-\delta'$ is contained in the orbit of $\delta$ under $\pi_1(J^\ast(\C),\sigma)$.
\item[(b)] The action of $\pi_1(J^\ast(\C),\sigma)$ on the complexification $(\mathbb V_{\sigma})_\C$ is irreducible.
\end{enumerate}
\end{prop}

\begin{proof}
Part (a) may be reduced to verifying two topological facts:
\begin{itemize}
\item[(i)] Viewing $J$ as a line in $S$, so that
\[
J^\ast = J\setminus R \subseteq S \setminus R,
\]
the induced map $\pi_1(J^\ast(\C),\sigma)\to \pi_1(S(\C)\setminus R(\C),\sigma)$ is surjective.
\item[(ii)] Choose any two points $p_0,p_1\in J\cap R=J\cap \hat R$ and let $\gamma_i$ be an elementary path corresponding to $p_i$, for $i=0,1$.  (That is, $\gamma_i$ starts at $\sigma$, travels a path in $J^\ast$ to a point near $p_i$, circles once around $p_i$ in the positive sense, and returns to $\sigma$ along its initial path; see \cite[Fig. 3]{Lamot}.)  Then the homotopy classes $[\gamma_0],[\gamma_1]\in \pi_1(S(\C)\setminus R(\C),\sigma)$ are conjugate.
\end{itemize}

The analogous statements to (i) and (ii) are true when $R$ is replaced by the irreducible hypersurface $\hat R$; see \cite[\S7.4, 7.5]{Lamot}.  On the other hand, considering the homomorphisms
\[
\pi_1(J^\ast(\C),\sigma)\to \pi_1(S(\C)\setminus R(\C),\sigma) \to \pi_1(S(\C)\setminus \hat R(\C),\sigma)
\]
induced by the inclusions $J^\ast \hookrightarrow S\setminus R\hookrightarrow S\setminus \hat R$, the second of these is an isomorphism, since $R\setminus\hat R$ has complex codimension at least two in $S$; see \cite[Prop. 4.1.1]{Dimca}.  Hence the statements (i) and (ii) regarding $R$ are true. 

Since the representation $\lambda$ factors as
\[
\lambda: \pi_1(J^\ast(\C),\sigma) \ \to \ \pi_1(S(\C)\setminus R(\C),\sigma) \ \to \  \bigO(\mathbb V_{\C,\sigma},\phi_\sigma)
\]
statement (i) shows that part (a) may be reduced to proving that either $\delta'$ or $-\delta'$ is contained in the orbit of $\delta$ under $\pi_1(S(\C)\setminus R(\C),\sigma)$.  Following the same reasoning as \cite[\S7.6]{Lamot}, one may then accomplish this using statement (ii), the Picard-Lefschetz formula, and the nondegeneracy of $\phi_\sigma$ on $\mathbb V_\sigma$.

Part (b) then follows from part (a) using (again) the Picard-Lefschetz formula and the nondegeneracy of $\phi_\sigma$; see \cite[p.46]{Lamot}.
\end{proof}

\begin{cor}\label{f_or_f_cor}
The image of the monodromy representation $\lambda$ in (\ref{rhoL}) is either finite or Zariski-dense in the complex algebraic group $\bigO((\mathbb V_{\sigma})_\C,\phi_\sigma)$.
\end{cor}

\begin{proof}
This follows by an application of \cite[Lemma $4.4.2^s$]{Del2} to the quadratic space $((\mathbb V_{\sigma})_\C,-\phi_\sigma)$, using Proposition \ref{single_orbit}(a) and the Picard-Lefschetz formula.
\end{proof}

\noindent\emph{Proof of Theorem A}.  As noted, it suffices to prove Theorem \ref{supp_thm}.  Let $M$ denote the Zariski-closure of the image of $\lambda$ and assume that $M$ is finite.  Since the action of $M$ on $(\mathbb V_{\sigma})_\C$ is irreducible by Propostion \ref{single_orbit}(b), this implies that any $M$-invariant nonzero bilinear form on $\mathbb V_\sigma$ is either positive or negative definite.  But $\phi_\sigma$ is such a form and has signature $(2,7)$, giving a contradiction.  Hence the theorem follows from Corollary \ref{f_or_f_cor} \qed

% 5555555555555555555555555

\section{Applications to Galois representations}\label{Tate_sec}

Throughout \S\ref{Tate_sec}, $k_0$ will be a finitely generated subfield of $\C$ and $k$ will be its algebraic closure in $\C$.  We will frequently replace $k_0$ by a finite extension when necessary, sometimes without mention.  This will not be problematic, since it is sufficient to prove the statements in Theorem B over some finite extension of the original field of definition.

We will prove Theorem B by following the axiomatic approach laid out in \cite{Andre-KS}.  However, a modification of Andr\'e's axioms is required in the present situation, and we discuss more carefully those points of the argument that are potentially affected by this modification.

\subsection{Axioms}\label{axiom_subsec}

Let $X$ be a smooth projective geometrically connected surface defined over $k_0$.  Suppose
\begin{equation}\label{A_ref}
\Omega\subseteq H_\Z:= H^2(X_\C,\Z)(1)/(\text{tors})
\end{equation}
denotes a fixed sublattice of classes such that:
\begin{itemize}
\item[(i)] $\Omega$ is spanned by the cycle classes of a finite collection of numerically independent effective divisors $D_1, D_2, \ldots, D_m$ on $X$ defined over $k_0$ (and hence all of the classes in $\Omega$ are algebraic).
\item[(ii)] The divisor $D_1$ is ample.
\end{itemize}
Suppose also that $H_\Z$ has Hodge numbers $h^{-1,1}=h^{1,-1}=1, h^{0,0}>0$, and $h^{p,q}=0$ otherwise.

If
\[
\theta: H_\Z\otimes H_\Z \To H^4(X_\C,\Z)(2) \tilde\To \Z
\]
is the bilinear form given by the cup product, we define the orthogonal complement
\begin{equation}\label{V_defn}
V_\Z := \Omega^\perp\subseteq H_\Z.
\end{equation}
As $\Omega$ contains an ample class, $(V_\Z,\theta)$ is an integral polarized Hodge structure.  If $R\subseteq \C$ is any (commutative, unital) ring of characteristic zero, let $V_R=V_\Z\otimes R$.  We will also set $V:=V_\Q$.

We now suppose that $\pi:\mathcal X\to\mathcal S$ is a smooth projective family of surfaces defined over $k_0$, with $\mathcal S$ smooth and geometrically connected, that satisfies four axioms (A1) through (A4).

\begin{enumerate}
\item[(A1)] There is  a point $s\in \mathcal S(k_0)$ such that $X$ is isomorphic to $\mathcal X_s$ over $k_0$.  After fixing such an isomorphism, we may assume that $X=\mathcal X_s$.  Let $\sigma = s_\C\in\mathcal S(\C)$.
\item[(A2)]  There exist effective divisors $\mathcal D_1, \mathcal D_2, \ldots,\mathcal D_m$ on $\mathcal X$, flat over $\mathcal S$, whose pullbacks to $X=\mathcal X_s$ are numerically equivalent to the divisors $D_1, D_2, \ldots, D_m$.  Moreover, the pullback of $\mathcal D_1$ to any fiber $\mathcal X_\tau$, $\tau\in\mathcal S(\C)$, is ample.
\end{enumerate}

The elements of $\Omega\subseteq H_\Z$ thus extend to global sections of $\mathbb H_\Z:=R^2(\pi^{\an}_\C)_\ast \Z(1)/(\text{tors})$, and one of these sections restricts to an ample class at every point.  If 
\[
\phi: \mathbb H_\Z \otimes\mathbb H_\Z \To R^4(\pi_\C^{\an})_\ast \Z(2)\simeq \Z.
\]
is the cup product form, then the orthogonal complement $\mathbb V_\Z$ of the global sections coming from $\Omega$ is a polarized integral variation of Hodge structure of weight zero with Hodge numbers
\begin{equation}\label{hdgV}
h^{-1,1}=h^{1,-1}=1, \qquad h^{0,0}=:N>0
\end{equation}
and $h^{p,q}=0$ otherwise.  For a ring $R$ of characteristic zero, let $\mathbb V_R:=\mathbb V_\Z\otimes R$, and set $\mathbb V:=\mathbb V_\Q$.  With $\sigma=s_\C\in\mathcal S(\C)$ as in (A1), we have $(\mathbb V_{\Z,\sigma},\phi_\sigma)=(V_\Z,\theta)$.
\begin{enumerate}
\item[(A3)] There exists $\mu\in\mathcal S(\C)$ such that the Hodge structure $\mathbb V_\mu$ contains nonzero algebraic classes. 
\item[(A4)]  The image of the monodromy representation
\[
\Lambda: \pi_1(\mathcal S(\C),\sigma)\to\bigO(\mathbb V_,\sigma,\phi_\sigma)=\bigO(V,\theta)
\]
contains a dense subgroup of $\SO(V,\theta)$.
\end{enumerate}

These axioms are similar to the those being considered in \cite[p.207]{Andre-KS} (which in turn are similar to axioms in \cite{Rap}), but they differ in two ways.  The first is that in (A2) we focus on a subvariation of the full primitive cohomology of the family, albeit one whose complement in $\mathbb H=R^2(\pi_\C^{\an})_\ast\Q(1)$ is still algebraic.  The second and more significant difference is represented in (A3) and (A4), which together replace the single assumption that the image of period map of the variation $\mathbb V$ over $\mathcal S_\C$ contains an open subset of the period domain.

Finally we note that, since we have restricted our attention to the simpler situation where all fibers of $\pi\colon \mathcal X\to \mathcal S$ are surfaces, the Lefschetz $(1,1)$ Theorem eliminates the need to assume that the Hodge classes in each fiber of $\mathbb V_\Z$ are algebraic.

\medskip

\noindent\textbf{Main Example.}  We recall the smooth projective family $\pi:\mathcal X\to\mathcal S$ over $\Q$ constructed in \S\ref{prelim_sec}, with $\mathcal S$ smooth, geometrically connected, and of dimension 5.  If $X$ is any admissible \cc surface defined over $k_0$, then by Corollary \ref{CCfamily} there exists a point $s\in \mathcal S(k_0)$ such that $X\simeq \mathcal X_s$.  Keeping the notation $H_\Z$ above, let $\Omega\subseteq H_\Z$ denote the span of the cycle classes of the canonical divisor and an Albanese fiber.  As in \S\ref{VHS_subsec}, these two divisor classes are arise from effective divisors on $\mathcal X$ that are flat over $\mathcal S$, one of which gives an ample divisor class on all fibers. Thus taking the orthogonal complement yields the variation of Hodge structure $\mathbb V_\Z$ over $S$ with Hodge numbers
\begin{equation}\label{CC_hodge}
h^{-1,1}=h^{1,-1} = 1, \quad h^{0,0} = 7.
\end{equation}
This verifies Axioms (A1) and (A2) for this example.  Lastly, Theorem \ref{CM_exists} and Theorem A show that (A3) and (A4) hold.

Note in this example that the image of the period map of the variation $\mathbb V_\Z$ over $\mathcal S$ cannot contain an open subset of the relevant period domain; indeed, by (\ref{CC_hodge}), the period domain has dimension 7, but the dimension of $\mathcal S$ is 5.

%%%%%%%%%%%%%%%%%%%%%%

\subsection{}\label{motives_subsec}

In the course of proving Theorem \ref{general_Tate} below (of which Theorem B is a special case), we will work with Andr\'e's theory of \emph{motivated cycles} developed in \cite{Andre-Mot}.  All motivated cycles are absolute Hodge cycles (in the sense of \cite{Del-AH}), and conjecturally both notions are equivalent to algebraic cycles modulo homological equivalence.

Given a subfield $E$ of $\C$, one may define the category of pure motives for motivated cycles over $E$ with coefficients in $\Q$; this is a Tannakian semisimple category with a fiber functor to the category of rational Hodge structures called the \emph{Betti realization}.  Below we will refer to this simply as the category of \emph{motives over $E$}.  Given two motives $\mathscr W_1$ and $\mathscr W_2$ over $E$ with Betti realizations $W_1$ and $W_2$, to say that a Hodge correspondence $c: W_1\to W_2$ is \emph{motivated over $E$} is to say that $c$ is the Betti realization of a morphism $\mathscr W_1\to\mathscr W_2$.  By Tannakian formalism, the category of motives over $E$ is equivalent to the category of finite-dimensional $\Q$-representations of the associated motivic Galois group $\mathcal G_E$.  

Let us mention some motives that will appear below.  For any variety $Y$ over $E$, we let $\mathscr H^i(Y)$ denote the motive whose Betti realization is $H^i(Y,\Q)$.  Considering the case that the field $E$ contains $k_0$, the variety $X_E$ gives (with a harmless ambiguity of notation) the motive $\mathscr H^2(X)$ over $E$.  Since algebraic cycles are motivated, we have a submotive of $\mathscr H^2(X)$ whose Betti realization is the algebraic subspace $\Omega\otimes \Q$ in $H^2(X,\Q)$: as these cycles are defined over $k_0$, this submotive is the sum of $m=\rank \Omega$ copies of the trivial motive over $E$.  As the cup product $H^2(X,\Q)\otimes H^2(X,\Q)\to H^4(X,\Q)$ is motivated, there is also a motive $\mathscr V$ with Betti realization $V=(\Omega\otimes\Q)^\perp$.  We note in particular that 
\begin{equation}\label{motive_decomp}
\mathscr H^2(X)(1)\simeq \mathbf{1} \oplus \cdots\oplus \mathbf{1} \oplus \mathscr V,
\end{equation}
where there are $m$ summands of the trivial motive $\mathbf{1}$ on the right.  In the case when $E=\C$, we will denote the motive $\mathscr V$ as $\mathscr V_\C$.  In general, we will use the script notation to denote the motives that are realized by certain cohomological objects.  For instance, $\mathscr C^+(\mathscr V)$ will denote the motive with Betti realization $C^+(V)$ (where $C^+(V)$ denotes the even Clifford algebra of the quadratic space $(V,\theta)$), and $\enmot(\mathscr H^i(Y))$ will denote the motive with Betti realization $\End(H^i(Y,\Q))$.

The idea of the proof of Theorem \ref{general_Tate} is to show (in Theorem \ref{AH_over_k}) that $\mathscr V$ belongs to the Tannakian category generated by the motives of abelian and zero-dimensional varieties over $k_0$, and then to exploit the work of Faltings \cite{Fal, Fal-Wus} on abelian varieties.

%%%%%%%%%%%%%%%%%%%%%%

\subsection{}\label{KS_subsec}

We now invoke the Kuga--Satake--Deligne construction, which associates abelian schemes to certain types of variations of Hodge structure, specifically those over smooth connected varieties with that are polarized of weight zero and have Hodge type $\set{(-1,1),(0,0),(1,-1)}$ with $h^{-1,1}=1$.  For more details about the construction, we refer to \cite{Del-K3, Andre-KS}.  Since $\mathbb V$ is of this type, we obtain:

\begin{thm}[Kuga--Satake \cite{KS}, Deligne \cite{Del-K3}]\label{KS_thm}
After perhaps replacing $\mathcal S$  by a connected finite \'etale covering, as well as the variation $\mathbb V_\Z$ over $\mathcal S_\C$ by the corresponding pullback variation, the following exist:
\begin{itemize}
\item[(a)] a complex abelian scheme $a:\mathcal A\to \mathcal S_\C$,
\item[(b)] a ring $B$ with an embedding $\nu:B\hookrightarrow \End_{\mathcal S_\C}(\mathcal A)$,
\item[(c)] an isomorphism of variations of Hodge structure
\begin{equation}\label{isom_of_rings}
u_\Z: C^+(\mathbb V_\Z) \tilde{\To} \underline{\End}_{B}(R^1 a^{\an}_\ast \Z)
\end{equation}
that induces an underlying isomorphism of local systems of rings.
\end{itemize}
\end{thm}

Henceforth, we assume that $\mathcal S$ has been replaced by a finite \'etale covering, and $\pi:\mathcal X\to\mathcal S$ by its respective pullback, so as to guarantee the existence of the structures in Theorem \ref{KS_thm} over $\mathcal S$.  Note that, since we are insensitive to replacing $k_0$ by a finite extension, we may take this finite covering to also be defined over $k_0$.  Also note that after making this replacement, the properties (A1) through (A4) still hold for $\pi:\mathcal X\to\mathcal S$.  A key result in what follows is:

\begin{lem}\label{unique_map}
Let $R$ be a subring of $\C$ and consider the isomorphism
\begin{equation}\label{R_isom_of_rings}
u_\Z \otimes 1: C^+(\mathbb V_R) \tilde{\To} \underline{\End}_{B}(R^1 a^{\an}_\ast R)
\end{equation}
obtained by tensoring (\ref{isom_of_rings}) with the constant sheaf $R$.  Then (\ref{R_isom_of_rings}) is the unique isomorphism from $C^+(\mathbb V_R)$ to $\underline{\End}_{B}(R^1 a^{\an}_\ast R)$ as local systems of $R$-algebras on $\mathcal S_\C$.
\end{lem}

\begin{proof}
The argument comes from \cite[Lemme 6.5.1]{Del-K3}.  Given two isomorphisms of local systems of $R$-algebras
\[
C^+(\mathbb V_R) \tilde\To \underline{\End}_B(R^1 a^{\an}_{\ast} R),
\]
one obtains an automorphism of the local system $C^+(\mathbb V_R)$, i.e., a $\pi_1(\mathcal S(\C),\sigma)$-invariant $R$-algebra automorphism of the even Clifford algebra $C^+(\mathbb V_{R,\sigma}) = C^+(V_{R})$; upon base change to $\C$, we then obtain a $\pi_1(\mathcal S(\C),\sigma)$-invariant $\C$-algebra automorphism of $C^+(V_{\C})$.  But by the density of the monodromy representation $\Lambda$ in $\SO(V_\C,\theta)$ from (A4), such an automorphism necessarily commutes with a dense subgroup of $\Spin(V_\C)$, which acts by conjugation on $C^+(V_\C)$.  Deligne \cite[Prop.\ 3.5]{Del-K3} shows that this implies the automorphism of $C^+(V_\C)$ is the identity, and hence the original two maps of local systems of $R$-algebras agree.
\end{proof}

Taking $R=\Q$ in Lemma \ref{unique_map}, one obtains the isomorphism of rational variations of Hodge structure that we denote simply as
\begin{equation}\label{u_isom}
u: C^+(\mathbb V) \tilde{\To} \underline{\End}_B(R^1 a^{\an}_\ast \Q).
\end{equation}
Note that under our identification $X=\mathcal X_s$, specializing $u$ at $\sigma=s_\C$ yields
\begin{equation}\label{u_isom_fiber}
u_{\sigma}: C^+(V) \tilde\To \End_B(H^1(\mathcal A_{\sigma},\Q)),
\end{equation}
which is an isomorphism of both weight zero Hodge structures and $\Q$-algebras.  One has similar isomorphisms at all other $\tau\in\mathcal S(\C)$.

\begin{prop}\label{AH_over_C}
The Hodge correspondence $u_{\sigma}$ in (\ref{u_isom_fiber}) is motivated over $\C$.
\end{prop}

\begin{proof}
One argues as in \cite[Prop.\ 6.2.1]{Andre-KS}.  The point is to show that the collection of all $\pi_1(\mathcal S(\C),\sigma)$-invariant elements in $\Hom\bigl(C^+(V),\End_B(H^1(A_\sigma,\Q))\bigr)$ is the Betti realization of a motive over $\C$, by using tensor constructions and Deligne's theorem of the fixed part.  Since the motivic Galois group $\mathcal G_\C$ preserves isomorphisms of $\Q$-algebras $C^+(V) \to \End_B(H^1(A_\sigma,\Q))$, this implies it will preserve such $\Q$-algebra isomorphisms that are also invariant under $\pi_1(\mathcal S(\C),\sigma)$.  By Lemma \ref{unique_map}, $u_\sigma$ is the only such isomorphism, so $u_\sigma$ is fixed by $\mathcal G_\C$.
\end{proof}

Recall the special point $\mu\in\mathcal S(\C)$ of (A3), and let us fix a subspace $W_\mu\subseteq \mathbb V_\mu$ of codimension one by taking the orthogonal complement some nonzero algebraic class.  Then $W_\mu$ is a rational Hodge structure with Hodge numbers $h^{-1,1}=h^{1,-1}=1$ and $h^{0,0}=N-1$.  One may construct the Kuga-Satake variety $\KS(W_\mu)$ of $W_\mu$ (defined up to isogeny since $W_\mu$ has rational coefficients), along with an isomorphism of Hodge structures
\begin{equation}\label{mu_isom}
v_\mu: C^+(W_\mu) \tilde\To \End_{B'}(H^1(\KS(W_\mu),\Q)),
\end{equation}
for a particular ring of endomorphisms $B'$ of $\KS(W_\mu)$.  The abelian variety $\KS(W_\mu)$ is isogenous to $\mathcal A_\mu^2$ (see \cite[4.1.5]{Andre-KS}).  Using this fact, Proposition \ref{AH_over_C}, and the fact that Hodge cycles on abelian varieties are motivated \cite{Andre-Mot}, one may show that the Hodge correspondence $v_\mu$ is also motivated.

\begin{lem}\label{trivial_det}
The motive $\det \mathscr V_\C$ is trivial.
\end{lem}

\begin{proof}
With $N = h^{0,0}(\mathbb V_\Z) = h^{0,0}(V)>1$ as in (\ref{hdgV}), there are two cases:

\smallskip

\noindent \emph{$N$ even.}  In this case the lemma follows fairly directly from Proposition \ref{AH_over_C}.  Arguing as in \cite[6.2.2]{Andre-KS}, one may consider the $\bigO(V,\theta)$-stable increasing filtration $F_j$ on $C^+(V)$ by the images of the spaces $\bigoplus_{i=0}^j V^{\otimes 2i}$.  Then the $\bigO(V,\theta)$-modules $\det V$ and $F_{N/2}/F_{(N-1)/2}$ are isomorphic, and one may use this to identify $\det V$ noncanonically with a subspace of $C^+(V)$ that is stabilized by $\bigO(V,\theta)$.  Then, as the motivic Galois group $\mathcal G(\mathscr V_\C)$ of $\mathscr V_\C$ is a subgroup of $\bigO(V,\theta)$, this identifies $\det \mathscr V_\C$ with a submotive of the motive $\mathscr C^+(\mathscr V_\C)$, and hence with a submotive $\mathscr M$ of $\enmot_B (\mathscr H^1(\mathcal A_\sigma))$ by Proposition \ref{AH_over_C}.   Let $M\subseteq \End(H^1(\mathcal A_\sigma,\Q))$ denote the Betti realization of $\mathscr M$.  Since $\bigO(V,\theta)$, and hence $\mathcal G(\mathscr V_\C)$, must act on $\mathscr M\simeq \det \mathscr V_\C$ by $\set{\pm 1}$, it follows from the (Kuga-Satake) construction of $\mathcal A_\sigma$ that the Hodge group of $H^1(\mathcal A_\sigma,\Q)$ acts trivially on $M$; thus the classes in $M$ are algebraic and $\mathscr M$ is the trivial motive.

\smallskip

\noindent\emph{$N$ odd.}  In this case one proceeds along the lines of \cite[6.4.1]{Andre-KS}.  Since the Hodge structure $W_\mu$ has even dimension $N-1$, one may use the motivated correspondence $v_\mu$ in (\ref{mu_isom}) to argue as in the previous case that $\det W_\mu$ is the realization of the trivial motive, and hence so is $\det \mathbb V_\mu$.  Considering the local system $\det \mathbb V$, one then applies a deformation argument to conclude that $\det \mathbb V_\tau$ is the realization of the trivial motive at all other $\tau\in\mathcal S(\C)$.  Since the realization of $\det \mathscr V_\C$ is $\det \mathbb V_\sigma$, this gives the lemma in this second case.
\end{proof}

\begin{prop}\label{AH_embed}
Over $\C$, there is a motivated correspondence
\begin{equation}\label{AH_gamma_C}
\gamma: V \hookrightarrow \End (H^1(\mathcal A_{\sigma},\Q)).
\end{equation}
\end{prop}

\begin{proof}
As in the previous proof, we have two separate cases:

\smallskip

\noindent \emph{$N$ odd}.  Following the proof of \cite[6.2.3]{Andre-KS}, one may form (as in the ``$N$ even'' portion of the proof of Lemma \ref{trivial_det}) the filtration $F_j$ on $C^+(V)$ given by the images of the spaces $\bigoplus_{i=0}^j V^{\otimes 2i}$.  The $\bigO(V,\theta)$-module $F_{(N+1)/2}/F_{(N-1)/2}$ (which is isomorphic as a vector space to $\bigwedge^{N+1} V$) is isomorphic to $V\otimes \det V$, and this allows one to (noncanonically) identify $\mathscr V_\C\otimes \det\mathscr V_\C$ as a submotive of $\enmot_B (\mathscr H^1(\mathcal A_\sigma))$.  But $\det \mathscr V_\C$ is the trivial motive by Lemma \ref{trivial_det}, giving the result in this case.

\smallskip

\noindent \emph{$N$ even.}  Arguing as in \cite[6.4.3]{Andre-KS}, one may show that there is a $\SO(V,\theta)$-invariant embedding $V\hookrightarrow \End(C^+(V))$.  (Note that $\GSpin(V)$ acts upon $C^+(V)$ by left multiplication and hence acts through its quotient $\SO(V,\theta)$ upon $\End(C^+(V))$; this is the intended action of $\SO(V,\theta)$ on the right side of the aforementioned embedding.)  Explicitly, let $d=N+2$, choose an orthogonal basis $\set{e_1,\ldots, e_{d}}$ of $V$, and choose a nonisotropic vector $v$: then one may check that
\[
\beta: \quad w \mapsto L_{we_1\ldots e_d} R_v
\]
gives an $\SO(V,\theta)$-invariant embedding $\beta: V\hookrightarrow \End(C^+(V))$, where $L_a$ and $R_a$ denote multiplication on the left and right by $a$, respectively.  (See the proof of \cite[6.4.3]{Andre-KS} for a more conceptual view of such an embedding.)  From the Kuga--Satake--Deligne construction, it follows that $\beta$ extends to an embedding of variations of Hodge structure
\[
\underline\beta: \mathbb V \hookrightarrow \underline{\End} (R^1 a^{\an}_\ast \Q).
\]
It suffices to show that $\beta = \underline\beta_\sigma$ is motivated and, by a deformation argument, this will in turn follow if one shows that $\underline{\beta}_\mu$, with $\mu\in\mathcal S(\C)$ as in (A3), is motivated.

Let $v^\ast$, $e_1^\ast, \ldots, e_d^\ast$ in $\mathbb V_\mu$ be the continuations of $v,e_1,\ldots, e_d$ in $V=\mathbb V_\sigma$ along some chosen homotopy class of paths from $\sigma$ to $\mu$, so that
\[
\underline{\beta}_\mu: \quad w \mapsto L_{we_1^\ast\ldots e_d^\ast} R_{v^\ast}.
\]
(Note that $v^\ast$, $e_1^\ast, \ldots, e_d^\ast$ depend upon the choice of path, but $\underline{\beta}_\mu$ does not.)  We may suppose without loss of generality that $e_d^\ast$ is algebraic and that $\mathbb V_\mu = W_\mu\oplus \Q e_d^\ast$.  Let us write three maps:
\begin{eqnarray*}
\delta_1: && \mathbb V_\mu = W_\mu\oplus \Q e_d^\ast \To C^+(W_\mu)\oplus \Q e_d^\ast \\
&& (w,\alpha e_d^\ast) \longmapsto (we_1^\ast \ldots e_{d-1}^\ast ,\alpha e_d^\ast) \\
\delta_2: && C^+(W_\mu)\oplus \Q e_d^\ast \To \End (C^+(W_\mu)) \oplus \Q e_d^\ast \simeq \End (H^1(\KS(W_\mu),\Q)) \oplus \Q e_d^\ast \\
 && (a,\alpha e_d^\ast) \longmapsto (L_a,\alpha e_d^\ast) \\
\delta_3: && \End(C^+(W_\mu)) \oplus \Q e_d^\ast \To \End(C^+(\mathbb V_\mu)) \simeq \End (H^1(\mathcal A_u,\Q)) \\
&& (\varphi, \alpha e_d^\ast) \longmapsto (\varphi+\alpha L_{e_d^\ast}) L_{e_1^\ast\ldots e_d^\ast} R_{v^\ast} 
\end{eqnarray*}
Then one may check that $\underline\beta_\mu = \delta_3\delta_2\delta_1$.  The map $\delta_1$ is essentially the composition $W_\mu\to W_\mu\otimes \det W_\mu \to C^+(W_\mu)$, which is motivated: the first of these is motivated since $\det W_\mu$ is motivically trivial by Lemma \ref{trivial_det}, and the second is motivated following an argument similar to the case of ``$N$ odd" above.  The map $\delta_2$ is in fact $v_\mu\oplus\text{Id}$ (with $v_\mu$ as in (\ref{mu_isom})), hence is motivated.  Finally, one may verify that $\delta_3$ is a morphism of $\SO(W_\mu)$-modules, and hence gives a Hodge correspondence between $\End (H^1(\KS(W_\mu),\Q)) \oplus \Q e_d^\ast$ and $\End (H^1(\mathcal A_u,\Q))$; since Hodge cycles on abelian varieties are motivated, it follows that $\delta_3$ is motivated.  Therefore $\underline\beta_\mu$ is motivated, as desired.
\end{proof}

\begin{thm}\label{AH_over_k}
The abelian variety $\mathcal A_{\sigma}$ has a model $A$ over $k_0$ such that the inclusion $\nu_\sigma: B\hookrightarrow \End(\mathcal A_\sigma)$ descends to $B\hookrightarrow \End_{k_0}(A)$.  Furthermore, the motivated correspondence $\gamma$ in (\ref{AH_gamma_C}) descends to a motivated correspondence over $k_0$.
\end{thm}

\begin{proof}
We follow \cite[\S5.5]{Andre-KS}, which in turn is a stronger version of \cite[Prop.\ 6.5]{Del-K3}.  Without reproducing the entire proof, we take care to explain the role played by (A4), the density of the monodromy.

Consider the collection $\mathcal C_1=(\mathcal S,\mathcal X,\pi,\mathcal A,a, \nu)$; as this collection is defined by a finite number of equations, there is  (after perhaps replacing $k_0$ by a finite extension) a smooth connected variety $T$ over $k_0$ such that $\mathcal C_1$ descends to a collection $\mathcal C_2=(\mathcal S_2,\mathcal X_2,\pi_2, \mathcal A_2, a_2, \nu_2)$ over the function field $k_0(T)$ of $T$.  Note that as $\pi:\mathcal X\to\mathcal S$ is defined over $k_0$, $\pi_2:\mathcal X_2\to\mathcal S_2$ is obtained simply by base change to $k_0(T)$:
\[
\mathcal S_2=\mathcal S_{k_0(T)}, \quad \mathcal X_2=\mathcal X_{k_0(T)}, \quad \pi_2=\pi_{k_0(T)}.
\]
In fact, upon replacing $T$ if necessary, the collection $\mathcal C_2$ is the generic fiber of a collection $\mathcal C_3=(\mathcal S_3,\mathcal X_3,\pi_3,  \mathcal A_3, a_3, \nu_3)$ defined over all of $T$.   Just as before, the first three objects in $\mathcal C_3$ are obtained by base change to $T$:
\[
\mathcal S_3=\mathcal S_T, \quad \mathcal X_3=\mathcal X_T, \quad \pi_2=\pi_T.
\]

To show the existence of $A$, the main point is to produce an isomorphism of local systems of rings over $(\mathcal S_3)_\C$ that is similar to $u_\Z$ in (\ref{isom_of_rings}).  This is not automatic, since the transcendental isomorphism $u_\Z$ does not ``spread'' to $T$ along with the collection $\mathcal C_1$.  To obtain this we use (A4), in the guise of Lemma \ref{unique_map}.

For a prime number $\ell$, one tensors $u_\Z$ with $\Z_\ell$ and uses comparison to obtain an isomorphism of $\Z_\ell$-sheaves of algebras
\[
u_\ell: C^+({\mathbb V}_\ell) \tilde{\To} \underline{\End}_B (R^1 a_{\ast} \Z_\ell) 
\]
in the \'etale topology on $\mathcal S_\C$; here ${\mathbb V}_\ell$ is the subsheaf of the $\Z_\ell$-sheaf $R^2 (\pi_\C)_\ast \Z_\ell(1)$ obtained as the complement of the global sections arising from the algebraic classes $\Omega$ in (A2) (i.e., the construction is exactly similar to that of $\mathbb V_\Z$).  Then $u_\ell$ automatically descends to an isomorphism of \'etale sheaves over $\bar{\mathcal S}_2:=\mathcal S_2\otimes_{k_0(T)} \overline{k_0(T)}$, where $\overline{k_0(T)}$ is the algebraic closure of $k_0(T)$.  This means that the map on fibers
\begin{equation}\label{tau_isom}
u_{\ell,\sigma}\colon C^+({\mathbb V}_{\ell,\sigma}) \tilde\To \End(H^1(\mathcal A_{2,\sigma},\Z_\ell))
\end{equation}
is invariant under the action of the geometric \'etale fundamental group $\pi_1(\bar{\mathcal S}_2,\sigma)$.  But, after choosing a ring embedding $\Z_\ell\hookrightarrow \C$, it follows from Lemma \ref{unique_map} that (\ref{tau_isom}) is the unique isomorphism of $\Z_\ell$-algebras between $C^+({\mathbb V}_{\ell,\sigma})$ and $\End(H^1(\mathcal A_{2,\sigma}),\Z_\ell))$ that is invariant under $\pi_1(\bar{\mathcal S}_2,\sigma)$.  Since an element of the arithmetic \'etale fundamental group $\pi_1(\mathcal S_2,\sigma)$ sends such an isomorphism to another, it must also fix (\ref{tau_isom}), allowing one to conclude that $u_\ell$ in fact descends to an isomorphism of \'etale sheaves over $\mathcal S_2$ itself.

Hence, after perhaps replacing $T$ again, $u_\ell$ is the generic fiber of an isomorphism of \'etale sheaves over $T$ that, by comparison, yields an isomorphism of local systems of rings 
\begin{equation}\label{analytic_isom_over_T}
C^+(\mathbb V_T)\tilde{\To} \underline{\End}_B(R^1 a_3^{\an}\Z)
\end{equation}
 in the analytic site on $(\mathcal S_3)_\C=\mathcal S_\C\times_\C T_\C$; here $\mathbb V_T$ denotes the pullback of $\mathbb V$ from $\mathcal S_\C$ to $\mathcal S_\C\times_\C T_\C$.  Then, as in \cite[Lemma 5.5.1]{Andre-KS}, one uses (\ref{analytic_isom_over_T}) to prove that any specialization of the abelian scheme $\mathcal A_3\to S_T=S\times_{k_0} T$ to a point along the subvariety $s\times_{k_0} T$ (with $s\in\mathcal S(k_0)$ as in (A1)) in fact gives a model for $(\mathcal A_{\sigma},\nu_\sigma)$ over the residue field of that point.  In particular, choosing a $k_0$-valued point of $s\times _{k_0}T$, we get a model $A$ for $\mathcal A_{\sigma}$ over $k_0$ for which the endomorphisms $B\hookrightarrow \End(\mathcal A_\sigma)$ descend to endomorphisms $B\hookrightarrow \End_{k_0}(A)$, completing the first part of the theorem.

Having established this, one considers the motive $\enmot(\mathscr{H}^1(A))$ over $k_0$, whose Betti realization is $\End(H^1(A_\C,\Q)) = \End(H^1(\mathcal A_\sigma,\Q))$.  Proposition \ref{AH_embed} gives an embedding $\gamma$ of $V$ into $\End(H^1(A_\C,\Q))$ which is motivated over $\C$.  This descends automatically to a motivated embedding over the algebraic closure $k$ of $k_0$, and hence to a finite extension of $k_0$, which we may assume to be $k_0$ itself.  Thus one has an embedding
\begin{equation}\label{AH_gamma}
\gamma_0: \mathscr V \hookrightarrow \enmot(\mathscr{H}^1(A))
\end{equation}
of motives over $k_0$.
\end{proof}

\subsection{}\label{ThmB_subsec}

Theorem B is now a special case of the following result:

\begin{thm}\label{general_Tate}
In the axiomatic setting of \S\ref{axiom_subsec}, fix a prime number $\ell$ and consider the $\ell$-adic representation
\[
r_\ell:\Gal(k/k_0)\to \Aut\bigl( H^2(X_{k},\Q_\ell)(1) \bigr).
\]
Then the following hold:
\begin{itemize}
\item[(i)] The representation $r_\ell$ is semisimple.
\item[(ii)]  (Tate Conjecture) Let $V_{\alg}$ be the $\Q_\ell$-subspace generated by the image of the cycle class map
\[
c_\ell:\CH^1(X_{k})\to H^2(X_{k},\Q_\ell)(1).
\]
Then $V_{\alg}$ is exactly the subspace of elements in $H^2(X_{k},\Q_\ell)(1)$ that are fixed by an open subgroup of $\Gal(k/k_0)$.
\end{itemize}
\end{thm}

\begin{proof}
As noted earlier, it suffices to prove both of the statements after replacing $k_0$ by a finite extension, and so we may assume the results of \S\ref{KS_subsec}.

The Galois representation $H^2(X_{k},\Q_\ell)(1)$ is the $\ell$-adic realization of the motive $\mathscr H^2(X)$.  Letting $V_\ell$ denote the $\ell$-adic realization of $\mathscr V$, the decomposition (\ref{motive_decomp}) shows that the $\Gal(k/k_0)$-module $H^2(X_{k},\Q_\ell)(1)$ is the direct sum of $V_\ell$ and $m$ copies of the trivial representation.  The truth of part (i) of the theorem now follows from the semisimplicity of $V_\ell$, which in turn follows from taking the $\ell$-adic realization of the motivated correspondence $\gamma_0$ (\ref{AH_gamma}) proved in Theorem \ref{AH_over_k} and using Faltings' proof of the semisimplicity conjecture for abelian varieties \cite{Fal, Fal-Wus}.

For part (ii), one must show the elements of $V_\ell$ fixed by an open subgroup of $\Gal(k/k_0)$ are $\Q_\ell$-combinations of algebraic (i.e., divisor) classes.  First let us set up notation.  For a smooth projective variety $Y$ over $k_0$, let $i_\ell: H^\ast(Y_\C,\Q)\otimes_\Q \Q_\ell \tilde\To H^\ast(Y_k,\Q_\ell)$ be the canonical comparison $\Q_\ell$-isomorphism obtained from our fixed embedding $k\hookrightarrow \C$.  Then the motivated embedding $\gamma_0$ over $k_0$ in (\ref{AH_gamma}) yields a commutative diagram
\begin{equation}\label{comp_diag}
\xymatrixcolsep{5pc}
\xymatrix{
V \otimes_\Q \Q_\ell \ar[d]_{i_\ell} \ar@{^{(}->}[r]^-{\gamma \otimes 1} & \End (H^1(A_\C,\Q))\otimes_\Q \Q_\ell \ar[d]_{i_\ell} \\
V_\ell \ar@{^{(}->}[r]^-{\gamma_\ell} & \End (H^1(A_k,\Q_\ell)), \\
}
\end{equation}
in which the bottom arrow $\gamma_\ell$ is an embedding of $\Gal(k/k_0)$-modules.

Now suppose that $v_\ell\in V_\ell$ is fixed by an open subgroup of $\Gal(k/k_0)$; then so is $\gamma_\ell(v_\ell)$ and so, by Faltings' proof of the isogeny conjecture \cite{Fal, Fal-Wus}, $\gamma_\ell(v_\ell)$ arises from a $\Q_\ell$-combination of isogenies of $A_k$.  It then follows that $(\gamma\otimes 1)(i_\ell^{-1}(v_\ell)) = i_\ell^{-1}(\gamma_\ell(v_\ell))$ is a $\Q_\ell$-combination of Hodge classes.  
But as $\gamma$ is a Hodge correspondence, any $\Q_\ell$-combination of Hodge classes in the upper right of (\ref{comp_diag}) that lies in the image of $\gamma\otimes 1$ necessarily comes from such a combination in  the upper left.  Thus $i_\ell^{-1}(v_\ell)$ is a $\Q_\ell$-combination of Hodge classes, and hence by the Lefschetz (1,1) Theorem is a $\Q_\ell$-combination of divisor classes.  Therefore $v_\ell$ is as well.
\end{proof}

\bibliography{CC_arith_published}{}
\bibliographystyle{math}

\end{document}